\documentclass[11pt,twoside]{amsart}
\UseRawInputEncoding
\usepackage{mathrsfs}
\usepackage{txfonts}
\usepackage{bbm}
\usepackage{amsfonts}
\usepackage{amsmath}
\usepackage{amssymb}
\usepackage{wasysym}
\usepackage{psfrag}
\usepackage{graphics,epsfig,amsmath}  
\usepackage{comment}
\usepackage{enumerate}
\usepackage{color}
\usepackage{enumerate}
\usepackage{fancyhdr}
\usepackage{diagbox}
\usepackage{threeparttable}
\usepackage{tabularx}
\usepackage{booktabs}

\newtheorem{Theorem}{Theorem}[section]
\newtheorem{Lemma}[Theorem]{Lemma}
\newtheorem{Definition}{Definition}[section]
\newtheorem{remark}[Theorem]{Remark}
\newtheorem{Example}[Theorem]{Example}
\newtheorem{Corollary}[Theorem]{Corollary}

\newtheorem{Proposition}{Proposition}[section]
\usepackage{metalogo}
\usepackage{color}
\usepackage{cite}
\allowdisplaybreaks

\selectfont

\title[Existence and regularity results for mixed elliptic equations]{Existence and regularity of weak solutions \\for mixed local and nonlocal semilinear \\ elliptic equations}

\author[F. Cheng]{Fuwei Cheng}
\address{School of Mathematical Sciences, Beijing Normal University, No. 19, XinJieKouWai St., HaiDian District, Beijing 100875, P. R. China}
\email{fwcheng@mail.bnu.edu.cn}

\author[X. Su]{Xifeng Su}
\address{School of Mathematical Sciences, Laboratory of Mathematics and Complex Systems (Ministry of Education)\\
	Beijing Normal University,
	No. 19, XinJieKouWai St., HaiDian District, Beijing 100875, P. R. China}
\email{xfsu@bnu.edu.cn, billy3492@gmail.com}

\author[J. Zhang]{Jiwen Zhang}
\address{School of Mathematical Sciences, Beijing Normal University, No. 19, XinJieKouWai St., HaiDian District, Beijing 100875, P. R. China}
\email{jwzhang628@mail.bnu.edu.cn}

\begin{document}

	\begin{abstract}
	We study the existence, multiplicity and regularity results of weak solutions for the Dirichlet problem of a semi-linear elliptic equation  driven by the mixture of the usual Laplacian and fractional Laplacian 
	\begin{equation*}
		\left\{%
		\begin{array}{ll}
			-\Delta u + (-\Delta)^{s} u+ a(x)\  u =f(x,u) & \hbox{in $\Omega$,} \\
			u=0 &  \hbox{in $\mathbb{R}^n\backslash\Omega$} \\
		\end{array}%
		\right.
	\end{equation*} 
	where $s \in (0,1)$, $\Omega \subset \mathbb{R}^{n}$ is a bounded domain, the coefficient $a$ is a function of $x$ and the subcritical nonlinearity $f(x,u)$ has superlinear growth at zero and infinity. 
	
	We  show the existence of a non-trivial weak solution by  Linking Theorem and Mountain Pass Theorem respectively  for  $\lambda_{1} \leqslant 0$ and $\lambda_{1} > 0$, where $\lambda_{1}$ denotes the first eigenvalue of $-\Delta + (-\Delta)^{s} +a(x)$. 
         In particular, adding a symmetric condition to $f$, we obtain infinitely many solutions via Fountain Theorem. 
	
	Moreover, for the regularity part,  we first prove  the $L^{\infty}$-boundedness of weak solutions and then establish up to  $C^{2, \alpha}$-regularity up to boundary.
\end{abstract}

\maketitle
\noindent{\it Keywords:}  Mountain Pass Theorem, Linking Theorem, variational methods, De Giorgi-Nash-Moser theory, regularity theory, mixed local and nonlocal elliptic equations.\bigskip

\tableofcontents

\section{Introduction}
In this article, we are concerned with the existence, multiplicity and regularity of weak solutions to the following mixed local and nonlocal elliptic problem with Dirichlet boundary condition
\begin{equation}\label{problem}
	\left\{%
	\begin{array}{ll}
		-\Delta u + (-\Delta)^{s} u+ a(x) u =f(x,u) & \hbox{in $\Omega$,} \\
		u=0 &  \hbox{in $\mathbb{R}^n\backslash\Omega$} \\
	\end{array}%
	\right.
\end{equation}
where $s \in(0,1), \Omega \subset \mathbb{R}^{n}$ is a bounded domain and 
\begin{equation}\label{Assumption of a}
	a(x)\in \begin{cases}
		L^{1}(\Omega) \quad & \text{if } n=1,\\
		L^{r}(\Omega),~ r >1 \quad & \text{if } n=2,\\
		L^{l/2}(\Omega), ~l\geqslant n \quad & \text{if } n \geqslant 3.
	\end{cases}
\end{equation}	
Here,
$(-\Delta)^{s}$ is the fractional Laplacian defined by a singular integral which coincides with Riesz derivative on the whole space
\[(-\Delta)^{s} u(x):=c(n,s) ~\mathrm{P.V.} \int_{\mathbb{R}^{n}} \frac{u(x)-u(y)}{|x-y|^{n+2 s}} d y,\]
where 
$c(n,s) >0$ is a suitable normalization constant, whose explicit value does not play a role here and $\mathrm{P.V.}$ stands for the Cauchy principal value. 

The mixed differential and pseudo-differential elliptic operators
\[\mathcal{L}=-\Delta+(-\Delta)^{s}, \quad \text { for some } s \in(0,1) \]
naturally arise in the study of superposition of Brownian motion and $2 s$-stable L\'evy process and have a wide range of concrete applications such as biological population dynamics (see \cite{DV21,DPV23,MPV13,PV18}), plasma physics (see \cite{BD13}), finance and control theory (see \cite{MP96}).

Recently, there is a great attention dedicated to theoretical studies of elliptic equations driven by $\mathcal{L}$, such as viscosity solution theory  \cite{EK05,GC08},
existence and non-existence theory  \cite{SVWZ24,RS15}, Harnack inequality and H\"older continuity  \cite{F09,CKSV12,GK22}, interior and boundary regularity  \cite{SE22,SVWZ25}.


Our first goal in this article is to show the existence of weak solutions (see Definition~\ref{def:weak solution})
for the mixed local and nonlocal  elliptic problem~\eqref{problem} driven by the modified operator $\mathcal{L}_{a}:=-\Delta+(-\Delta)^{s} + a(x)$, which is somewhat general in the literature.

Suppose the nonlinear term $f$ : $\bar{\Omega} \times \mathbb{R} \rightarrow \mathbb{R}$ is a subcritical Carath\'eodory function verifying the following conditions:
\begin{itemize}
	\item[\textbf{(C)}] $f$ is continuous in $\bar{\Omega} \times \mathbb{R}$;
	\item[\textbf{(H1)}] there exist $c_{f} >0$ and $q \in\left(2,2^{*}\right)$, such that
	\[	|f(x, t)| \leqslant  c_{f} (1+ |t|^{q-1}) \text { for a.e. } x \in \Omega, \ t \in \mathbb{R} ; \]
	\item[\textbf{(H2)}] $ \lim \limits_{t \rightarrow 0} \frac{f(x, t)}{t}=0 \text { uniformly for any } x \in \Omega $;
	\item[\textbf{(H3)}] $\lim \limits_{|t| \rightarrow \infty} \frac{F(x, t)}{t^{2}}=+\infty$ uniformly for any $x \in \Omega$;
	\item[\textbf{(H4)}] there exists $T_{0}>0$ such that for any $x \in \Omega$, the function
	\begin{center}
		$t \mapsto \frac{f(x, t)}{t}$ is increasing in $t > T_{0}$, and decreasing in $t < -T_{0}$.
	\end{center}
\end{itemize} 
Here we denote $F(x, t):=\int_{0}^{t} f(x, \tau) d \tau$ and the critical value
\[2^{*}:= \begin{cases}
	\frac{2n}{n-2}, \quad &N \geqslant 3,\\
	\infty, \quad &N=1,2.
\end{cases}\]

The strategy for existence proofs we take is based on several minimax theorems. That is, we will deal with the functional $\mathcal{J}:\mathcal{X}^{1,2}(\Omega)\to\mathbb{R}$ related to problem~\eqref{problem},
which is defined in \eqref{energyfunctional}  as 
\[\mathcal{J}(u)= \frac{1}{2}\|u\|_{\mathcal { X } ^ { 1 , 2 } ( \Omega )}^{2} + \frac{1}{2} \int_{\Omega}a(x)u^{2}dx - \int_{\Omega} F(x,u)dx.\]
Here, the function space $\mathcal{X}^{1,2}(\Omega)$ is given in Definition~\ref{definition of the working space} as the completion of $C_0^\infty(\Omega)$ with respect to the global norm
	\[
	\|u\|_{\mathcal{X}^{1,2}(\Omega)} = \left( \|\nabla u\|_{L^2(\mathbb{R}^n)}^2 + [u]_s^2 \right)^{1/2},
	\]
where $[u]_s$ denotes the standard Gagliardo seminorm in \eqref{Gagliardoseminorm}.
	
This functional is imposed to have a suitable geometric structure and to satisfy an a priori compactness condition.
More precisely, the assumptions (H1)-(H2) are to ensure the geometry of $\mathcal{J}$, while (H3)-(H4) are to guarantee  the compactness, which is a bit weaker than the standard Ambrosetti-Rabinowitz condition \cite{MR0370183}: 
\begin{itemize}
	\item [\textbf{(AR)}]
	there exist $\mu >2$ and $r>0$ such that a.e. $x \in \Omega$,  $t \in \mathbb{R},~~|t| \geqslant r$		
	\[0<\mu F(x, t) \leqslant t f(x, t).\]
\end{itemize} 

Consequently, the global existence theorem is obtained according to the different geometric properties of $\mathcal{J}$, i.e., we apply both Linking Theorem and Mountain Pass Theorem respectively  for  $\lambda_{1} \leqslant 0$ and $\lambda_{1} > 0$ where $\lambda_1$ is the first eigenvalue of $\mathcal{L}_{a}$.


\begin{Theorem}\label{Thm: LandMP}
	Let $f$ verify (C), (H1)-(H4). We have the following conclusions:
	\begin{description}
		\item [(1) $\lambda_{1} \leqslant 0$] assume in addition when $0\in [\lambda_k, \lambda_{k+1})$\medskip
		    \begin{itemize}
		        \item [\textbf{(P)}] $\lambda_{k} \frac{t^{2}}{2} \leqslant F(x,t)  \text {  for any } x \in \Omega,\  t \in \mathbb{R} $, 
                     \end{itemize}   where $\lambda _ { 1 } \leqslant \lambda _ { 2 } \leqslant \cdots \leqslant \lambda _ { k } \leqslant  \lambda _ { k + 1 } \leqslant \cdots$ are eigenvalues of problem~\eqref{eigen problem} and 
each eigenvalue is repeated according to its multiplicity, \medskip \newline
		        then problem~\eqref{problem} admits a non-trivial Linking solution $u \in \mathcal{X}^{1,2}(\Omega)$;	
		    \smallskip
		\item [(2) $\lambda_{1} > 0$] problem~\eqref{problem} admits a non-trivial Mountain Pass solution $u \in \mathcal{X}^{1,2}(\Omega)$.
	        \end{description}
\end{Theorem} 

 We remark that 
 \begin{itemize}
\item Assumption (P) provides the linking structure. Theorem~\ref{Thm: LandMP} can be seen as a mixed local and nonlocal counterpart of local problem \cite[Theorem 2.18]{Willembook} and nonlocal problem \cite[Theorem 1]{SV13}. 

\item $f(x,u)$ satisfying (H1) does not mean that $f_{a}(x,u):= - a(x)u + f(x,u)$ satisfies (H1). So  the present result cannot be covered by that obtained in \cite{SVWZ24}.

\item  when $\lambda_{1} > 0$ and $a(x)$ does not change sign, one can find a non-trivial non-negative (non-positive) weak solution. In particular, while studying the non-negative solutions as in \cite{DSVZ25}, one can have 
 the symmetry properties of such solutions (see Theorem~\ref{Thm: symmetric}).
\end{itemize}

As an application of the well-known \textit{Fountain Theorem} (first established in \cite{Bartsch93}),
 by imposing
\begin{itemize}
	\item [\textbf{(S)}] 
	$f(x,-t)=-f(x,t)$ for any $x \in \Omega, t\in \mathbb{R}$,
\end{itemize}
 infinitely many weak solutions of \eqref{problem} are obtained below:
\begin{Theorem}\label{Thm: infinitely many sol Theorem}
	Assume $f$ satisfies (C), (H1), (H3), (H4) and (S). Suppose  $\lambda_{1} > 0$.
	Then, problem \eqref{problem} admits infinitely many weak solutions $\{u_{j}\}_{j \in \mathbb{N}} \subset \mathcal{X}^{1,2}(\Omega)$ such that $\mathcal{J}(u_j) \rightarrow +\infty$, as $j \rightarrow +\infty$.
\end{Theorem}
Thanks to the symmetry assumption (S), if $u$ is a weak solution of problem~\eqref{problem}, so is $-u$. Hence, our results actually assure the existence of infinitely many pairs $\{u_{j},-u_{j}\}_{j \in \mathbb{N}}$ of weak solutions. We also point out that, all of the above existence and multiplicity results are valid for a ``good" $a(x) \in L^{\infty}(\Omega)$ instead of \eqref{Assumption of a}.
\smallskip

Our next goal is to establish the regularity theory of weak solutions to problem~\eqref{problem}. 

We first use \textit{De Giorgi-Nash-Moser theory} to obtain the following two $L^{\infty}$-regularity theorems by a rather complete analysis on  $a(x) \not \equiv 0$ and
\begin{itemize}
\item $f=f(x,u)$  or
\item  $f=f(x)$.
\end{itemize}
It is worth noting that, when $a(x) \equiv 0$, $L^{\infty}$-regularity and interior (or boundary) regularity have been proved in \cite{SE22}  and  \cite{SVWZ25}  for the linear term $f(x)$ and the  nonlinearity $f(x,u)$ respectively.

Noticing that it is immediately to see the following continuous imbedding facts for dimensions 1 and 2 below:
\begin{equation}\label{imbedding when n=1,2}
	\begin{cases}
		\mathcal { X } ^ { 1 , 2 } ( \Omega ) \hookrightarrow L^{\infty}(\Omega) \quad & \text{ if } n=1,\\
		\mathcal { X } ^ { 1 , 2 } ( \Omega ) \hookrightarrow L^{p}(\Omega),\ 1\leqslant p < \infty \quad  & \text{ if } n=2,
	\end{cases}
\end{equation}
it suffices to  show the $L^\infty$-boundedness for dimension $n\geqslant 3$.

On the one side, we have
\begin{Theorem}\label{Thm: boundedness}
	Let $n \geqslant 3$ and  $\Omega \subset \mathbb{R}^{n}$ be an open bounded domain. Suppose $u \in \mathcal { X } ^ { 1 , 2 } ( \Omega )$ is a weak solution of
	\[-\Delta u + (-\Delta)^{s} u+ a(x) u =f(x,u)  \quad \text{in }\Omega. \]
	Assume that there exist $c_{f} >0$ and $q \in [2,2^{*}]$ such that
	\begin{equation}\label{H1*}
		|f(x, t)| \leqslant  c_{f} \left(1+ |t|^{q-1} \right) 
		\quad \text { for a.e. } x \in \Omega,\  t \in \mathbb{R}.
	\end{equation} 
	If either of the following conditions holds:
	\begin{itemize}
		\item [(1)] $0 \leqslant a(x) \in L^{\frac{l}{2}}(\Omega)$, for some $l \geqslant n$;\smallskip
		\item [(2)] $a(x) \in L^{\infty}(\Omega)$,
	\end{itemize}
	then $u \in L^{\infty}(\Omega)$.
	Moreover, there exists a constant $C_{0}>0$, such that 
	\begin{equation*}
		\| u \| _ { L ^ { \infty } ( \Omega ) } \leqslant C_0  \left( 1 + \int _ { \Omega } | u | ^ { 2 ^ { * } \beta _ { 1 } } d x \right) ^ { \frac { 1 } { 2 ^ { * } \left( \beta _ { 1 } - 1 \right) } } ,  
	\end{equation*}
	where
	\[ C_0  : = \left\{ \begin{array} { l l } 
		C_0(n, \Omega, c_{f}) & \text{if (1) holds},  \\ 
		C_0(n, \Omega, c_{f}, |a|_{\infty} ) & \text{if (2) holds}  \end{array} \right. 
	\text{ and} \quad
	\beta _ { 1 } : = \left\{ \begin{array} { l l } 
		( 2 ^ { * } + 1 )/2 & \text{if (1) holds},  \\ 
		2 ^ { * } /2 & \text{if (2) holds.} \end{array} \right. \]
\end{Theorem} 

Note that compared to the assumption (H1), $q$ can be chosen to be $2$ or $2^*$ here.

On the other side, when $f$ depends only on $x$, we obtain an  $L^{\infty}$-regularity of weak solutions to problem~\eqref{problem}. 
\begin{Theorem}\label{Thm: boundedness2}
Let $n \geqslant 3$ and $ \Omega \subset \mathbb{R}^{n}$ be an open bounded domain. 
	Suppose $u \in \mathcal { X } ^ { 1 , 2 } ( \Omega )$ is a weak solution of
	\[-\Delta u + (-\Delta)^{s} u+ a(x) u =f(x)  \quad \text{in }\Omega. \]
	Assume $a(x), f(x) \in L^{l}(\Omega)$ for some $l>n/2$.
	Then $u \in L^{\infty}(\Omega)$.
\end{Theorem}

Once the $L^{\infty}$-regularity is obtained, interior $ C^{2, \alpha}$-regularity  can be obtained naturally by mollifier technique and cutoff argument as in \cite{SVWZ25}.
\begin{Theorem}
	Suppose $u \in \mathcal { X } ^ { 1 , 2 } ( \Omega )$ is a bounded weak solution of 
	\[-\Delta u + (-\Delta)^{s} u+ a(x) u =f(x,u)  \quad \text{in }\Omega, \]
	where $a(x) \in  L^{\infty}(\Omega) \cap C^{\alpha}_{loc}(\Omega)$ and $f(x,t) \in C^{\alpha}_{loc}(\Omega \times \mathbb{R})$. Assume $V$ is an open domain with $ V \subset \subset \Omega$. Then, $u \in C^{2, \alpha}(\bar{V})$ for any $\alpha \in (0, 1)$.
\end{Theorem}

We point out that, in order to obtain up to $C^{2,\alpha}$-regularity of weak solutions, it is natural to assume the coefficient function $a(x)$ has a better regularity, namely
 $a(x) \in  L^{\infty}(\Omega) \cap C^{\alpha}_{loc}(\Omega)$ instead of $a(x) \in L^{\frac{l}{2}}(\Omega), l \geqslant n$.\smallskip
 
Furthermore, using the H\"{o}lder estimate of $(-\Delta)^{s}u$ and the regularity theory of weak solutions to local problem driven by $- \Delta $, we then obtain $C^{2, \alpha}$-regularity up to boundary by \textit{continuity method}.
\begin{Theorem}\label{Thm: C2-regularity up to boundary}
	Let $s\in(0, 1/2)$ and $\alpha \in (0,1)$ be such that $\alpha + 2s \leqslant 1$.
	Assume $\partial \Omega$ is of class $C^{2, \alpha}$. Suppose $u \in \mathcal { X } ^ { 1 , 2 } ( \Omega )$ is a weak solution of \eqref{problem}. If $a(x) \in C^{\alpha}(\bar\Omega) $ and $f \in C^{\alpha}(\bar\Omega \times \mathbb{R})$ satisfies (H1), then $u \in C^{2, \alpha}(\bar\Omega)$.
\end{Theorem}

We remark that the restriction \(s \in (0, 1/2)\) and $\alpha\in(0,1)$ satisfying \(\alpha + 2s \leqslant 1\) in Theorem \ref{Thm: C2-regularity up to boundary} is sharp. We give a detailed explanation in Remark \ref{remark: restriction s}. 

The paper is organized as follows. In section~\ref{Sec:Preliminaries}, we  collect some elementary results of $\mathcal { X } ^ { 1 , 2 } ( \Omega )$, introduce the functional setting (such as weak solutions and energy functional) and deal with  some properties of an eigenvalue problem of $\mathcal{L}_{a}$. 

In section~\ref{Sec: Global existence results}, we obtain the existence of a non-trivial weak solution by both Linking Theorem and Mountain Pass Theorem   for  $\lambda_{1} \leqslant 0$ and $\lambda_{1} > 0$ respectively.
In particular, after imposing symmetry condition on the nonlinearity,  we obtain infinitely many weak  solutions using Fountain Theorem. 

In section \ref{Sec: Regularity of weak solution }, we use De Giorgi-Nash-Moser theory to have the global boundedness of weak solutions according to various conditions on the coefficient $a(x)$. Moreover, we improve their regularity  to $C^{2, \alpha}$ up to boundary and give some symmetry properties of the solutions.

\section{Preliminaries}	\label{Sec:Preliminaries} \setcounter{equation}{0}
In this section, we provide several  preliminary facts and results which will be useful in the sequel.

\subsection{The variational framework} Let us start by introducing the basic functional setting
to problem~\eqref{problem}.

Let $s \in(0,1)$. If $u:\mathbb{R}^{n} \rightarrow{R}$ is a measurable function, we set
\begin{equation}\label{Gagliardoseminorm}
[u]_{s}:= \left(\int_{\mathbb{R}^{n}} \int_{\mathbb{R}^{n}} \frac{(u(x)-u(y))^{2}}{|x-y|^{n+2 s}} d x d y \right)^{\frac{1}{2}} 
\end{equation}
which is \textit{Gagliardo seminorm} of order $s$. Fractional Sobolev space $H ^ { s }(\mathbb{ R }^{n}) $ is defined by
\[ H^{s}(\mathbb{ R }^{n}) =\left\{ u\in L^{2}(\mathbb{ R }^{n}): [u]_{s}^{2} <\infty \right\}.\]

If $u \in H ^ { s } ( \mathbb { R } ^ { n } ) ,$ then there is a relation between  $( - \Delta ) ^ { s }u$ and $[u]_{s}$ :
\begin{equation}\label{seminorm}
	\left[u\right]_{s}^{2} = 2c(n,s)^{-1}\left\| (- \Delta )^{\frac{s}{2}}u \right\|_{L^{2}\left(\mathbb{R}^{n}\right)}^{2}.
\end{equation}
 See for example \cite[Proposition 3.6]{DNPV12}.

After the above preparations, we now define an appropriate function space which is close related to the Dirichlet problem~\eqref{problem}.
\begin{Definition}[Function space]\label{definition of the working space}
	Given a bounded open set $ \Omega \subseteq \mathbb { R } ^ { n }$, we define the function space
	$\mathcal { X } ^ { 1 , 2 } ( \Omega )$ as the completion of $C _ { 0 } ^ { \infty } ( \Omega )$ with respect to the global norm 
	\[ \|u\|_{\mathcal { X } ^ { 1 , 2 } ( \Omega )} =\left( \| \nabla u \| _ { L ^ { 2 } \left( \mathbb { R } ^ { n } \right) } ^ { 2 } + [ u ] _ { s } ^ { 2 } \right) ^ { 1 / 2 } , \quad u \in C _ { 0 } ^ { \infty } ( \Omega ) .\]
\end{Definition}

It is easy to see  $\| \cdot \|_{\mathcal { X } ^ { 1 , 2 } ( \Omega )}$ is induced by a mixed local and nonlocal inner product
\[B _ { s } ( u , v ) : = \int_{ \mathbb { R } ^ { n } }  \nabla u \cdot \nabla v ~ d x + \int_{\mathbb{R}^{n}} \int_{\mathbb{R}^{n}} \frac { (u ( x ) - u ( y ) )( v ( x ) - v ( y ) ) } { | x - y | ^ { n + 2 s } } d x d y \] 
and $\mathcal { X } ^ { 1 , 2 } ( \Omega )$ is a \textit{ Hilbert space}. Observe that $B _ { s } ( u , v )$ is a bilinear mapping.

We then give some useful equivalent characterizations of  $\mathcal{X}^{1,2}(\Omega)$.
\begin{Proposition}\label{Prop: X12 equivalent characterization} 
	The space $\mathcal{X}^{1,2}(\Omega)$ has the following equivalent characterization:
	\begin{equation*}\label{X12}
		\begin{aligned}
			 \mathcal{X}^{1,2}(\Omega) 
			&= \overline{C_{0}^{\infty}(\Omega)}^{\|\cdot\|_{H^{1}(\mathbb{R}^{n})}}
			= \left \{ u \in H ^ { 1 } ( \mathbb { R } ^ { n } ) : u | _ { \Omega } \in H _ { 0 } ^ { 1 } ( \Omega )\text{ and } u \equiv 0 \text{ a.e. in } \mathbb{R}^{n} \backslash \Omega\right \}\\
			&= \left \{ u \in L ^ { 2 ^ { * } } ( \mathbb { R } ^ { n } ) : u \equiv 0 \text{ a.e. in } 
			\mathbb { R } ^ { n } \backslash \Omega , \nabla u \in L ^ { 2 } ( \mathbb { R } ^ { n } )\text{ and }  [ u ] _ {s} < \infty \right \}. 
		\end{aligned}
	\end{equation*}
\end{Proposition} 
\begin{proof}
	Note that  $u$ identically vanishes outside $\Omega$, and the  $L ^ { 2 }$-norm of $\nabla u$ on the whole of $\mathbb{R}^{n}$ is just the same as that restricted to $\Omega$. 
	Proposition~\ref{Prop: X12 equivalent characterization} follows from the \textit{continuous embedding} of $H^{1}(\mathbb{R}^{n})$ into
	$H^{s}(\mathbb{R}^{n})$ (see \cite[Proposition 2.2]{DNPV12}) and the classical
	\textit{Sobolev Poincar\'e inequality}.
\end{proof}

Since $\|u\|_{\mathcal{X}^{1,2}(\Omega )} \simeq\|\nabla u\|_{L^{2}(\Omega)}$ for all $ u \in \mathcal{X}^{1,2}(\Omega )$, we deduce the following proposition by \textit{Sobolev-Rellich imbedding theorem}.
\begin{Proposition}\label{Prop: continuous and compact embedding} The embedding $\mathcal {X}^{1,2}(\Omega) \subset L^{2^{*}}(\Omega)$ is continuous; the embedding $\mathcal {X}^{1,2} (\Omega)\subset L^{m}(\Omega),~m \in [1,2^*)$ is compact.
\end{Proposition}

We now give the definition of weak solutions to problem~\eqref{problem}. 
\begin{Definition}\label{def:weak solution}
	We say that $u \in \mathcal{X}^{1,2}(\Omega)$ is a weak solution of problem~\eqref{problem}  if 
	\begin{equation}\label{weak solution}
		B _ { s } ( u , \phi ) +  \int_{\Omega} a(x)\ u\  \phi \ dx = \int_{\Omega} f(x, u)\  \phi \ dx,
	\end{equation}
	for every test function $\phi \in \mathcal{X}^{1,2}(\Omega)$. 
\end{Definition} 
\begin{remark}
	The Definition~\ref{def:weak solution} is well posed. That is,
	\begin{itemize}
		\item [(i)] Owing to the Green's formula and the relation \eqref{seminorm} between  $( - \Delta ) ^ { s }u$ and $[u]_{s}$, it is easy to check 
		\begin{equation*}
			\begin{aligned}
				&\quad \int_{\mathbb{R}^{n}} \int_{\mathbb{R}^{n}} \frac { (u ( x ) - u ( y ) )( \phi ( x ) - \phi ( y ) ) } { | x - y | ^ { n + 2 s } } d x d y\\
				&=\int_{\mathbb{R}^{n}} (-\Delta)^{s}u(x)~ \phi(x) d x  =\int_{\mathbb{R}^{n}} (-\Delta)^{s/2}u(x)~ (-\Delta)^{s/2}\phi (x) d x \\
				&\leqslant \|(-\Delta)^{s/2}u(x)\|_{L^{2}(\mathbb{ R }^{n})}~\|(-\Delta)^{s/2} \phi (x)\|_{L^{2}(\mathbb{ R }^{n})}\\
				&=2^{-1}c(n,s) ~[u]_{s}[\phi]_{s}<+\infty.
			\end{aligned}
		\end{equation*}
		\item [(ii)] Thanks to ${X}^{1,2}(\Omega) \hookrightarrow L^{2^{*}}(\Omega )$ and the assumption \eqref{Assumption of a} of $a(x)$, we have
		\[(a(x)u,\phi )_{L^2(\Omega)}:= \int_{\Omega} a(x) \ u\  \phi \ dx
	 <+\infty.\]
		\item [(iii)] Since $f(x,u)$ satisfies the assumption (H1) with $q\in(2,2^{*})$, by H\"{o}lder inequality, we have 
		\begin{equation*}
				 \quad \int_{\Omega} f(x, u) \phi dx \leqslant c_{f} \int_{\Omega} (1+|u|^{q-1}) |\phi| dx  \leqslant c_{f} ( |\Omega|^{1/2}|\phi|_{2}+ |u|_{q}^{q-1}|\phi|_{q})<+\infty.
		\end{equation*}
		 Here and in the sequel, we denote $\|\cdot\|_{L^{p}(\Omega)}$ by $|\cdot|_{p}$.
	\end{itemize}
\end{remark}

Finally, one can observe that weak solutions of problem \eqref{problem} can be found as  critical points of the energy functional $\mathcal{J} : \mathcal{X}^{1,2}(\Omega) \rightarrow \mathbb{R}$ defined by
\begin{equation}\label{energyfunctional}
\mathcal{J}(u)= \frac{1}{2}\|u\|_{\mathcal { X } ^ { 1 , 2 } ( \Omega )}^{2} + \frac{1}{2} \int_{\Omega}a(x)u^{2}dx - \int_{\Omega} F(x,u)dx.
\end{equation}
It is easy to check that $\mathcal{J} \in C^1(\mathcal{X}^{1,2}(\Omega),\mathbb{R})$, and 
\begin{equation}\label{1}
	\langle  \mathcal{J}^{\prime}(u), \phi \rangle
	= B _ { s } ( u , \phi ) +  (a(x)u,\phi )_{L^2(\Omega)} - \int_{\Omega} f(x, u) \phi dx 
\end{equation}
for all $\phi \in \mathcal{X}^{1,2}(\Omega)$.

\subsection{Eigenvalue problem of $-\Delta + (-\Delta)^{s} +a(x)$}
We deal with the weak eigenvalue problem associated to $\mathcal{L}_{a}$ and give the following variational proposition. 
\begin{Proposition}\label{Prop:eigen}
	The weak eigenvalue problem 
	\begin{equation}\label{eigen problem}
		\left\{\begin{array}{l}
				B_{s}(u,\phi)+ \int_{\Omega} a(x)u \phi d x = \lambda \int_{\Omega} u \phi dx, \quad \forall \phi \in \mathcal{X}^{1,2}(\Omega) \\
				u \in \mathcal{X}^{1,2}(\Omega)
			\end{array}\right.
	\end{equation}
	\begin{itemize}
		\item [(i)] admits an eigenvalue
		\[\lambda _ { 1 } : = \inf \left\{ \|u\|_{\mathcal { X } ^ { 1 , 2 } ( \Omega )}^{2} + \int _ { \Omega }a ( x ) u ^ { 2 }  d x : u \in \mathcal{X}^{1,2}(\Omega),\|u\|_{L^{2}(\Omega)}=1  \right\}> - \infty .\]
		and there exists a non-trivial function $e_{1} \in \mathcal{X}^{1,2}(\Omega)$ such that $\left\|e_{1}\right\|_{L^{2}(\Omega)}=1$, which is an eigenfunction corresponding to $\lambda_{1}$, attaining the minimum;
		\item[(ii)] possesses a divergent sequence of eigenvalues $\left\{\lambda_{k}\right\}_{k \in \mathbb{N}}$ with 
		\[- \infty < \lambda_{1} \leqslant \lambda_{2} \leqslant \cdots \leqslant \lambda_{k} \leqslant \lambda_{k+1} \leqslant \cdots\]
		and
		$\lambda_{k} \rightarrow+\infty \text { as } k \rightarrow \infty $.
		Moreover, for any $k \in \mathbb{N}$, the eigenvalues can be characterized as follows:
		\begin{equation}\label{eigenvalue}
			\lambda_{k+1}=\min \left\{\|u\|_{\mathcal { X } ^ { 1 , 2 } ( \Omega )}^{2} + \int_{ \Omega }a ( x ) u ^ { 2 }dx: u \in P_{k+1} ,\|u\|_{L^{2}(\Omega)}=1\right\} ,
		\end{equation}
		where
		\[P_{k+1}:=\left\{u \in \mathcal{X}^{1,2}(\Omega) \text { s.t. } \int_{ \Omega } u e_{j} dx=0 \quad \forall j=1, \ldots, k\right\}.\]
		\end{itemize}
	\end{Proposition}
	To prove this proposition, we just show the following lemma, which is a first step to prove Proposition~\ref{Prop:eigen}. The rest is similar to that in \cite[Proposition 9]{SV13} and we omit it. 
\begin{Lemma}\label{Lem: eigen}
		Let $\mathcal{F}: \mathcal{X}^{1,2}(\Omega) \rightarrow \mathbb{R}$ be the functional defined as 
		\[\mathcal{F}(u)=\frac{1}{2} \left(\|u\|_{\mathcal{X}^{1,2}(\Omega)}^{2}+ \int_{ \Omega } a(x)u^{2}dx \right).\]
	    Let $X_*$ be a weakly closed non-trivial subspace of $\mathcal{X}^{1,2}(\Omega)$ and $\mathcal{M}_*:=\{u \in X_*:|u|_{2}=1\}$. Then there exists $u_* \in \mathcal{M}_*$ such that
		\begin{equation}\label{65}
			-\infty < \min _{u \in \mathcal{M}_*} \mathcal{F}(u)=\mathcal{F}\left(u_*\right),
		\end{equation}	
		and 
		\begin{equation}\label{61}
			B_{s}(u_{*},\phi)+ \int_{ \Omega } a(x)u_{*}\phi=\lambda_*\int_{\Omega}u_*(x)\phi(x)dx, \quad \forall \phi \in X_*,
		\end{equation}
	    where $\lambda_*= 2 \mathcal{F}(u_{*})$.
\end{Lemma}

In order to prove Lemma~\ref{Lem: eigen}, we need first to gain a weak continuous property of the map  $\mathcal{G} : u \in \mathcal { X }^ { 1 , 2 } ( \Omega )   \mapsto  \int _ { \Omega } a ( x ) u ^ { 2 } d x$.
That is, 
\begin{Lemma}\label{Lemma: weakly continuous}
	If $a $ satisfies \eqref{Assumption of a}, then the map $\mathcal{G}$ 
	is weakly continuous. 
\end{Lemma}
\begin{proof}
	Observe the embedding property \eqref{imbedding when n=1,2}, it is enough to focus on the case when the dimension $n \geqslant 3$. Thanks to $\mathcal{X}^{1,2} (\Omega) \hookrightarrow L^{2^{*}}(\Omega) $ and H\"{o}lder inequalities, 
	the map $\mathcal{G}$ is well defined.
	
	Assume that $u _ { j } \rightharpoonup u$ in $\mathcal { X } ^ { 1 , 2 }(\Omega)$ and consider an arbitrary subsequence $(v_{j})$ of $(u_{j})$. Since $\mathcal { X } ^ { 1 , 2 }(\Omega)\hookrightarrow \hookrightarrow L^{2}(\Omega)$,  going if necessary to a subsequence, we have
	\begin{equation*}
		\begin{aligned}
			v_{j} \rightarrow u  \text{ in } L^{2}(\Omega) \quad \text{ and } \quad 
			v _ { j } \rightarrow u \text{ a.e. on } \Omega
		\end{aligned}
	\end{equation*}
	as $j \rightarrow \infty$ and there exists $h \in L^{2}(\Omega)$ such that 
	\[ |v_{j}(x)| \leqslant h(x) \quad \text{ a.e. in } \mathbb{R}^{n} \text{ for any } j\in \mathbb{N}.\]
	Since $( v _ { j } ) \subset \mathcal{X}^{1,2} (\Omega)$ is bounded in $L ^ { 2 ^ { * } }(\Omega)$, $ ( v _ {j} ^ { 2 } )$ is bounded in $L ^ { n / ( n - 2 ) }(\Omega)$. Hence  $v _ { j } ^ { 2 } \rightharpoonup u ^ { 2 }$ in Hilbert space $L ^ { n / ( n - 2 ) }(\Omega)$. Noticing that the dual space of $L^{n/2}(\Omega)$ is $L ^ { n / ( n - 2 ) }$, we have 
	$\mathcal{G}(v_{j}) \rightarrow \mathcal{G}(u) $ by the \textit{Dominated Convergence Theorem}.
\end{proof}
\begin{proof}[\bf{Proof of Lemma~\ref{Lem: eigen}}]
		Consider a minimizing sequence  $(v_{j}) \subset \mathcal{X}_{*}:$
		\[ \frac{\|v_{j}\|_{\mathcal{X}^{1,2}(\Omega)}^{2}+\int_{ \Omega } a(x)v_{j}^{2}dx}{2|v_{j}|_{2}} \rightarrow \inf _{\mathcal{M}_*}\mathcal{F}(u) \quad \text{ as } j\rightarrow \infty. \]
		Let $w_{j}=\frac{v_{j}}{\|v_{j}\|_{\mathcal{X}^{1,2}(\Omega)}}$, then $\|w_{j}\|_{\mathcal{X}^{1,2}(\Omega)} = 1$ and 
		\[ \frac{1+\mathcal{G}(w_{j})}{2|w_{j}|_{2}} \rightarrow \inf _{\mathcal{M}_*}\mathcal{F}(u) \quad \text{ as } j\rightarrow \infty.\]
		Since $(w_{j})$ is bounded in $\mathcal{X}^{1,2}(\Omega)$, up to a subsequence, still defined by $(w_{j})$, there exists $w \in \mathcal{X}^{1,2}(\Omega)$ such that 
		\[	w_{j} \rightharpoonup w \text { in } \mathcal{X}^{1,2}(\Omega) \text{ and } w_{j} \rightarrow w \text { in } L^{2}(\Omega) .\]
	    It follows from Lemma~\ref{Lemma: weakly continuous} that $\mathcal{G}(w_{j}) \rightarrow \mathcal{G}(w)$. Since $w \neq 0$,
	    \[\inf _{\mathcal{M}_*}\mathcal{F}(u) = \lim_{j\rightarrow \infty} \frac{1+ \mathcal{G}(w_{j})}{|w_{j}|_{2}} \geqslant  \frac{1+ \mathcal{G}(w)}{|w|_{2}} > -\infty.\]
		
		Let $u_{j}=\frac{w_{j}}{|w_{j}|_{2}} \in \mathcal{M}_{*}$. Since $(u_{j})$ is bounded in $\mathcal{X}^{1,2}(\Omega)$, up to a subsequence, still defined by $(u_{j})$, there exists $u_{*} \in \mathcal{M}_*$ such that
		\[	u_{j} \rightharpoonup u_{*} \text { in } \mathcal{X}^{1,2}(\Omega) \text{ and } u_{j} \rightarrow u_{*} \text { in } L^{2}(\Omega) .\] 
		According to \textit{Fatou Lemma} and Lemma~\ref{Lemma: weakly continuous}, we deduce that
		\[\inf _{u \in \mathcal{M}_*} \mathcal{F}(u) = \lim \limits_{j \rightarrow \infty} \mathcal{F}(u_{j}) \geqslant
		\mathcal{F}\left(u_*\right) \geqslant \inf _{u \in \mathcal{M}_*} \mathcal{F}(u),\]
		which implies \eqref{65}. 
		
		Since $u_{*}$ is a constrained minimizer of the functional $\mathcal{F}$, by the \textit{Lagrange Multiplier Rule}, \eqref{61} is verified. Moreover,  $\lambda_*= 2 \mathcal{F}(u_{*})$.
		In fact, let $\varepsilon \in(-1,1), v \in X_*, u_{\varepsilon}=\frac{u_{*}+\varepsilon v}{|u_{*}+\varepsilon v|_2}$, then $u_{\varepsilon} \in \mathcal{M}_*$ and 
		\[\begin{aligned}
			& \quad 2\mathcal{F}\left(u_{\varepsilon}\right)= B_{a_s}( u_{\varepsilon},u_{\varepsilon}) =B_{s}( u_{\varepsilon},u_{\varepsilon}) + \int_{\Omega}a(x)u_{\varepsilon}^{2}~ dx\\
			& =\frac{ 1-2 \varepsilon \int_{\Omega} u_{*} v d x+\varepsilon^{2}|v|_{2}^{2}}{1-4 \varepsilon^{2}\left(\int_{\Omega} u_{*} v d x\right)^{2}+2 \varepsilon^{2}|v|_{2}^{2}+\varepsilon^{4}|v|_{2}^{4}} \\
			&\quad \cdot \left( \| u_{*} \|_{\mathcal { X } ^ { 1 , 2 } ( \Omega )}^{2}+ \int_{\Omega} a(x)u_{*}^{2}+2 \varepsilon  B_{a_s}( u_{*}, v) +\varepsilon^{2} (\| v \|_{\mathcal { X } ^ { 1 , 2 } ( \Omega )}^{2} + \int_{\Omega} a(x) v^{2}) \right)\\
			& \leqslant  \frac{1}{\left(1-\varepsilon^{2}|v|_{2}^{2}\right)^{2}}\left(2 \mathcal{F}\left(u_{*}\right)+2 \varepsilon (B_{a_s}(u_{*},v) -2 \mathcal{F}\left(u_{*}\right) \int_{\Omega} u_{*} v )+o(\varepsilon)\right),
		\end{aligned}\]
		where the last inequality is from $\int_{\Omega} u_* v dx \leqslant \left| u_*\right|_2|v|_2=|v|_{2}.$ Here we denote $B_{a_s}(u,v):= B_{s}(u,v) + (a(x)u, v)_{L^{2} (\Omega)}$. 
		
	    The minimality of $u_*$ implies \eqref{61}. 
\end{proof}

We now give some notations. For any $k \in \mathbb{N}$, we define
\begin{equation}\label{space:YZ}
	Y_{k}:=\operatorname{span}\left\{e_{1}, \ldots, e_{k}\right\},\quad 
	Z_{k}:=\overline{\operatorname{span}\left\{e_{k}, e_{k+1}, \ldots\right\}} 
\end{equation}
where $e_{i}$ is the eigenfunction corresponding to $\lambda_{i}$, attaining the minimum in \eqref{eigenvalue}, that is
\begin{equation}\label{9}
	|e_{i}|_{2}=1 \quad \text{ and } \quad \|e_{i}\|_{\mathcal{X}^{1,2}(\Omega)}^{2}+\int_{ \Omega }a ( x ) e_{i} ^ { 2 }dx= \lambda_{i}.
\end{equation}
Since $Y_{k}$ is finite-dimensional, all norms on $Y_{k}$ are equivalent. Therefore, there exist two positive constants $C_{k, q}$ and $\tilde{C}_{k, q}$, depending on $k$ and $q$, such that for any $u \in Y_{k}$
\begin{equation}\label{finite dim eq norm}
	C_{k, q}\|u\|_{\mathcal{X}^{1,2}(\Omega)} \leqslant\|u\|_{L^q(\Omega)} \leqslant \tilde{C}_{k, q}\|u\|_{\mathcal{X}^{1,2}(\Omega)} .
\end{equation}

\section{Global existence and multiplicity results}\label{Sec: Global existence results} \setcounter{equation}{0}
In this section, we apply Linking Theorem and Mountain Pass Theorem   for  $\lambda_{1} \leqslant 0$ and $\lambda_{1} > 0$ respectively, to show the existence of a non-trivial weak solution of equation \eqref{problem}.
To use variational methods, the functional $\mathcal{J}$ is required to satisfy a suitable geometric structure and some compactness condition such as Palais-Smale compactness condition (i.e., every Palais-Smale sequence of $\mathcal{J}$ has a convergent subsequence).

To obtain  the $(PS)_{c}$ condition of $\mathcal{J}$, we first give the following lemma.

\begin{Lemma}\label{Lemma: minimax A2}
	\cite[Theorem A.2]{Willembook}
	Assume that $|\Omega| < \infty$, $1 \leqslant p,r < \infty$, $ f \in \mathcal{C}(\bar{\Omega} \times \mathbb{R})$ and $ |f(x,u)| \leqslant c(1 + |u|^{p/r})$. Then, for every $u \in L^{p}(\Omega)$, $f(\cdot,u) \in L^{r}(\Omega)$ and the operator
	\[A:~ L^{p}(\Omega) \rightarrow L^{r}(\Omega), \  u \longmapsto f(x,u)\]
	is continuous.
\end{Lemma}

\begin{Proposition}\label{Prop:PSc}
	Let $a(x)$ satisfy \eqref{Assumption of a}, $f(x,t)$ satisfy (C), (H1), (H3)-(H4). Then
	\begin{itemize}
		\item [$(a)$] every Palais-Smale sequence of $\mathcal{J}$ is bounded in $\mathcal{X}^{1,2}(\Omega)$;
		\item [$(b)$] every Palais-Smale sequence of $\mathcal{J}$ has a convergent subsequence in $\mathcal{X}^{1,2}(\Omega)$.	
	\end{itemize}
\end{Proposition}
\begin{proof} 
	We just consider the case $n \geqslant 3$.
	
	Note that once $f$ satisfies  (H3)-(H4), $f_{a}(x,t) := -a(x) t +f(x,t)$ also satisfies (H3)-(H4). Then part $(a)$ follows from a standard contraposition argument(see e.g. \cite{SVWZ24}). 
	However, since $a \in L^{ \frac{l}{2} }(\Omega)$ may be unbounded, $f$ satisfying condition (H1) does not mean that $f_{a}$ satisfies condition (H1). We adopt a method different from the proof in \cite{SVWZ24} to demonstrate part $(b)$.
	
	Let $(u_{j})$ be a bounded Palais-Smale sequence in $\mathcal{X}^{1,2}(\Omega)$ such that 
	\begin{equation}\label{T' convergent to 0-2}
		\langle \mathcal{J}^{\prime}(u_{j}), \varphi \rangle \rightarrow 0,\quad  \forall \varphi \in \mathcal{X}^{1,2}(\Omega)
	\end{equation} 
	as $j \rightarrow \infty$.
	Since  $\mathcal{X}^{1,2}(\Omega)$ is a Hilbert Space, up to a subsequence, still denoted by $(u_{j})$, there exists $u_{\infty} \in \mathcal{X}^{1,2}(\Omega)$ such that
	\[\begin{array}{ll}
		u_{j} \rightharpoonup u_{\infty}  \text { in } \mathcal{X}^{1,2}(\Omega) \quad \text{and} \quad 
		u_{j} \rightarrow u_{\infty}  \text { in } L^{q}(\Omega ),  q \in (2, 2^{*})
	\end{array}\]
	as $j \rightarrow+\infty$. 
	 
	Note $B_{s}(u,v)$ is bilinear. Observe that
	\[\begin{aligned}
		 \| u_{j} -u_{\infty} \|_{\mathcal{X}^{1,2}(\Omega)}^{2} = & ~\langle \mathcal{J}^{\prime}(u_{j}) - \mathcal{J}^{\prime}(u_{\infty}), u_{j}-u_{\infty} \rangle 
		 \\& + \int_{\Omega} (f(x,u_{j}) - f(x,u_{\infty}))(u_{j} - u_{\infty})~dx - \int_{\Omega} a(u_{j} - u_{\infty})^{2}~dx .
	\end{aligned}\]
	By (H1), for every $ u \in L^{q}(\Omega)$, 
	\[ |f(x,u)| \leqslant  c_{f} (1+ |u|^{q-1}) =  c_{f} (1+ |u|^{ \frac{q}{q/(q-1)} }) . \]
	Applying Lemma~\ref{Lemma: minimax A2}, we have $f(x,u_{\infty}) \in L^{q/(q-1)} (\Omega)$ and 
	\[ f(x,u_{j}) \rightarrow f(x,u_{\infty}) \quad \text{ in } L^{q/(q-1)}(\Omega), \]
	as $j \rightarrow \infty$. Thus, 
	\[ \int_{\Omega} (f(x,u_{j}) - f(x,u_{\infty}))(u_{j} - u_{\infty}) ~dx \leqslant |f(x,u_{j}) - f(x,u_{\infty})|_{\frac{q}{q-1}} |u_{j}-u_{\infty}|_{q} \rightarrow 0. \]
	Together with \eqref{T' convergent to 0-2} and Lemma~\ref{Lemma: weakly continuous}, we have 
	$\| u_{j} -u_{\infty} \|_{\mathcal{X}^{1,2}(\Omega)} \rightarrow 0$.  
\end{proof}
We now show that $\mathcal{J}$ indeed possesses suitable geometric structure.
\subsection{$\lambda_{1}  \leqslant  0$: Linking type solution}
Since $\lambda_{1}  \leqslant 0$, we put the number 0 between two adjacent unequal eigenvalues
\[ \lambda_{1} \leqslant \lambda_{2} \leqslant \cdots \leqslant \lambda_{k} \leqslant 0 < \lambda_{k+1} \leqslant \cdots \quad \text{ for some } k \in \mathbb{N},\]
where $\lambda_{k}$ is the $k$-th eigenvalue of the operator $\mathcal{L}_{a}$ defined in Proposition~\ref{Prop:eigen}.

\begin{Lemma}\label{Lemma: Zk+1 inf}
	If $a(x) $ satisfies \eqref{Assumption of a} and $\lambda_{k} \leqslant 0 < \lambda_{k+1} $, then 
	\[\varsigma_{k+1} :=  \inf _ {\substack{u \in Z_{k+1} \\\|u\|_{\mathcal{X}^{1,2}(\Omega)}=1}} \left\{ \|u\|_{\mathcal{X}^{1,2}(\Omega)}^{2} + \int_{\Omega} a(x)u^{2} dx \right\} >0.\]
\end{Lemma}
\begin{proof}
	By the definition of $\lambda_{k+1}$, on $Z_{k+1}$ we have
	\[\lambda_{k+1}|u|_{2}^{2} \leqslant \|u\|_{\mathcal{X}^{1,2}(\Omega)}^{2} + \int_{\Omega} a(x)u^{2} dx.\]
	Consider a minimizing sequence $(u_{j}) \subset Z_{k+1}:$
	\[ \|u_{j}\|_{\mathcal{X}^{1,2}(\Omega)}=1, \quad 1+ \mathcal{G}(u_{j}) \rightarrow \varsigma_{k+1}.\]
	Going if necessary to a subsequence, we may assume $u_{j} \rightharpoonup u$ in $\mathcal{X}^{1,2}(\Omega)$. By Lemma~\ref{Lemma: weakly continuous},
	\[\varsigma_{k+1}= \lim_{j\rightarrow \infty} \left\{ \|u_{j}\|_{\mathcal{X}^{1,2}(\Omega)}^{2}+\mathcal{G}(u_{j}) \right\} \geqslant \|u\|_{\mathcal{X}^{1,2}(\Omega)}^{2}+ \mathcal{G}(u) \geqslant \lambda_{k+1} |u|_{2}^{2}.\]
	If $u=0, \varsigma_{k+1}=1$ and if $u \neq 0, \varsigma_{k+1} \geqslant \lambda_{k+1} |u|_{2}^{2} >0.$
\end{proof}
\begin{Proposition}\label{Prop:Linking}
	Let $\lambda_{k} \leqslant 0 < \lambda_{k+1}$. Assume $a$ satisfies \eqref{Assumption of a}, $f$ satisfies (P), (H1)-(H3). Then, there exist $\rho >r >0$ and $z \in N:=\{u \in Z_{k+1} \text{ s.t. } \|u\|_{\mathcal{X}^{1,2}(\Omega)}= r\}$ such that
	\[\inf_{N} \mathcal{J}(u) > \max_{M_{0}}\mathcal{J}(u)\]
	where 
	$ M_{0}:=\left\{u=y+\omega z: \|u\|_{\mathcal{X}^{1,2}(\Omega)}= \rho ,  y\in Y_{k} \text{ and } \omega \geqslant 0 \right\} \cup \left\{ u \in Y_{k}: \|u\|_{\mathcal{X}^{1,2}(\Omega)}\leqslant \rho \right\}$.
\end{Proposition}
\begin{proof}
	We just consider the case $n \geqslant 3$ and proceed step by step. 
	
	\textbf{Step 1}. In this step, we prove that there exist $r , \beta>0$ such that $\inf_{N} \mathcal{J}(u)\geqslant \beta$.
	
	$f$ satisfying (H1) and (H2) implies that,  
	for any $\varepsilon>0$ there exists $\delta(\varepsilon) >0$ such that for a.e. $x \in \Omega$ and any $t \in \mathbb{R}$
	\begin{equation}\label{estimate of F}
		|F(x, t)| \leqslant \varepsilon|t|^{2}+\delta(\varepsilon)|t|^{q}.
	\end{equation} 
	From Proposition~\ref{Prop: continuous and compact embedding} and Lemma~\ref{Lemma: Zk+1 inf}, for any $u \in Z_{k+1}$
	\begin{equation}\label{important estimates}
		\begin{aligned}
			\mathcal{J}(u)&= \frac{1}{2} \left(\|u\|_{\mathcal{X}^{1,2}(\Omega)}^{2} +  \int_{\Omega} a(x)u^{2} dx \right) - \int_{\Omega} F(x,u)dx\\
			&\geqslant \frac{\varsigma_{k+1}}{2} \|u\|_{\mathcal{X}^{1,2}(\Omega)}^{2} - \varepsilon|u|_{2}^{2} - \delta(\varepsilon)|u|_{q}^{q}\\
			&\geqslant \frac{\varsigma_{k+1}}{2} \|u\|_{\mathcal{X}^{1,2}(\Omega)}^{2} - \varepsilon |\Omega|^{ 1- \frac{2}{2^{*}} }|u|_{2^{*}}^{2} - \delta(\varepsilon)|\Omega|^{ 1- \frac{q}{2^{*}}  }|u|_{2^{*}}^{q}\\
			&\geqslant \frac{\varsigma_{k+1}}{2} \|u\|_{\mathcal{X}^{1,2}(\Omega)}^{2} - \varepsilon|\Omega|^{ 1- \frac{2}{2^{*}} }C\|u\|_{\mathcal{X}^{1,2}(\Omega)}^{2} - \delta(\varepsilon)|\Omega|^{  1- \frac{q}{2^{*}} }C\|u\|_{\mathcal{X}^{1,2}(\Omega)}^{q}\\
			&= \|u\|_{\mathcal{X}^{1,2}(\Omega)}^{2} \left[\frac{\varsigma_{k+1}}{2} - \varepsilon|\Omega|^{1- \frac{2}{2^{*}} }C \right] - \delta(\varepsilon)|\Omega|^{  1- \frac{q}{2^{*}} }C\|u\|_{\mathcal{X}^{1,2}(\Omega)}^{q},
		\end{aligned}
	\end{equation}
	where the second inequality  uses the H{\"o}lder inequality.
	
	Taking $0<\varepsilon < \frac{\varsigma_{k+1}}{2C|\Omega|^{1- 2/2^{*}}}$, it easily follows that
	\[\begin{aligned}
		\mathcal{J}(u) \geqslant \alpha\|u\|_{\mathcal{X}^{1,2}(\Omega)}^{2}\left(1-\kappa\|u\|_{\mathcal{X}^{1,2}(\Omega)}^{q-2}\right)
	\end{aligned}\]
	for suitable positive constants $\alpha$ and $\kappa$. 
	Let $u \in Z_{k+1} $ be such that $\|u\|_{\mathcal{X}^{1,2}(\Omega)}= r>0$. Choose $r$ sufficiently small such that $1-\kappa r^{q-2}>0$. So that
	\[\inf_{N} \mathcal{J}(u) \geqslant \alpha r^{2}\left(1-\kappa r^{q-2}\right)=: \beta>0.\]
	
	\textbf{Step 2}.
	Take  $z  := r \frac{e_{k+1}}{\|e_{k+1}\|_{\mathcal{X}^{1,2}(\Omega)}} \in N$. We prove that there exists $\rho > r$ such that 
	\begin{equation}\label{12}
		\max_{M_{0}}\mathcal{J}(u)<0.
	\end{equation} 
	
	In fact, (H3) implies that,  for all $M>0$, there exists $C_{M}>0$ such that
	\begin{equation}\label{H3}
		F(x, t) \geq M t^{2}-C_{M}, \quad  \text{ for a.e. }  x \in \Omega,  t \in \mathbb{R} .
	\end{equation}
	So that, for any $u=y+\omega z \in Y_{k} \oplus \mathbb{R} z$ where $ \omega \geqslant 0$, we have
	\begin{equation*}
		\begin{aligned}
			\mathcal{J}(u)
			&\leqslant \frac{1}{2} \|u\|_{\mathcal{X}^{1,2}(\Omega)}^{2} + \frac{1}{2} |a(x)|_{\frac{n}{2}} |u|_{2^{*}}^{2}- M|u|_{2}^{2} + C_{M}|\Omega|\\
			&\leqslant \frac{1}{2} \|u\|_{\mathcal{X}^{1,2}(\Omega)}^{2} + \frac{1}{2} \tilde{C}_{k+1, 2^{*}}^{2}|a(x)|_{\frac{n}{2}}\|u\|_{\mathcal{X}^{1,2}(\Omega)}^{2}   -M C_{k+1,2}^{2}\|u\|_{\mathcal{X}^{1,2}(\Omega)}^{2}+C_{M}|\Omega|
		\end{aligned}
	\end{equation*}
	where the last inequality is deduced from~\eqref{finite dim eq norm}.
	Take \[M>\frac{2 + \tilde{C}_{k+1, 2^{*}}^{2}|a(x)|_{\frac{n}{2}}}{2C_{k+1,2}^{2}} >0.\]
	Then, $\mathcal{J}(u)\leqslant -\|u\|_{\mathcal{X}^{1,2}(\Omega)}^{2}+C_{M}|\Omega|. $
	
	Let $u= y + \omega z$ be such that $\|u\|_{\mathcal{X}^{1,2}(\Omega)}= \rho>0$. Choose $\rho$ big enough such that
	\begin{equation}\label{7}
		\max\{\mathcal{J}(u): u=y+\omega z \text{ s.t. } y\in Y_{k},~~ \|u\|_{\mathcal{X}^{1,2}(\Omega)}= \rho, ~~\omega \geqslant 0  \} <0.
	\end{equation}
	
	Moreover, for any $u \in Y_{k}$, $u$ can be characterized as 
	$u(x)=\sum_{i=1}^{k} u_{i} e_{i}(x),$
	with $u_{i} \in \mathbb{R}, i=1, \ldots, k$. Since eigenfunction sequence $\left\{e_{1}, \ldots, e_{k}, \ldots\right\}$ is an orthonormal basis of $L^{2}(\Omega)$, 
	$	\int_{\Omega}|u(x)|^{2} d x=\sum_{i=1}^{k} u_{i}^{2} |e_{i}|_{2}^{2}
	$ and  
	\[\int_{\mathbb{R}^{n}} \left(|\nabla\sum_{i=1}^{k} u_{i}e_{i}|^{2} + \int_{\mathbb{R}^{n}} \frac{|\sum_{i=1}^{k}u_{i} \left(e_{i}(x)-e_{i}(y)\right)|^{2}}{|x-y|^{n+2s}} dy + a \left( \sum_{i=1}^{k} u_{i}e_{i} \right)^{2} \right) dx = \sum_{i=1}^{k} u_{i}^{2} \lambda_{i} |e_{i}|_{2}^{2}.\]
    Test the eigenvalue equation~\eqref{eigen problem} for $e_{i}$ by test function $e_{j}$ for $j\neq i$, 
     \[ B_{s}(e_{i},e_{j})+ \int_{ \Omega } a(x)e_{i}e_{j}=\lambda_{i}\int_{\Omega} e_{i} e_{j}dx =0 .\]
     
	By assumption (P), we get
	\begin{equation*}
		\begin{aligned}
			\mathcal{J}(u) 
			&= \frac{1}{2} \sum_{i=1}^{k} u_{i}^{2} \lambda_{i} \int_{\Omega}  e_{i}^{2} dx-  \int_{\Omega} F(x, u) d x  \leqslant \frac{1}{2} \lambda_{k} \int_{\Omega} \sum_{i=1}^{k} u_{i}^{2}e_{i}^{2} - F(x, u) d x \\
			& = \frac{1}{2} \lambda_{k} \int_{\Omega} \left(\sum_{i=1}^{k} u_{i}e_{i}\right)^{2}dx -  \int_{\Omega}  F(x, u) d x = \int_{\Omega} \lambda_{k} \frac{u^{2}}{2}- F(x, u) d x 
			 \leqslant 0
		\end{aligned}
	\end{equation*}
	thanks to $\lambda_{i} \leqslant \lambda_{k}$ for any $i=1, \ldots, k$. 
	Together with \eqref{7}, \eqref{12} follows.
	
	By combining steps 1 and 2, the assertion of Proposition~\ref{Prop:Linking} follows.  
\end{proof}
Now we give the proof of Theorem \ref{Thm: LandMP} when $ \lambda_{1} \leqslant 0$.
\begin{proof}[\bf{Proof of Theorem \ref{Thm: LandMP} when $ \lambda_{1} \leqslant 0$}] Assume that $\lambda_{k} \leqslant 0 < \lambda_{k+1}$ for some $k\in \mathbb{N}$. Since the geometry of the Linking Theorem is assured by Proposition~\ref{Prop:Linking} and  $(PS)_{c}$ condition is obtained by Proposition~\ref{Prop:PSc}, we can exploit the Linking Theorem to find a critical point $u \in \mathcal{X}^{1,2}(\Omega)$ of $\mathcal{J}$. Furthermore,
	\[\mathcal{J}(u)\geqslant \inf_{N} \mathcal{J}(u) \geqslant \beta>0=\mathcal{J}(0)\]
	and so $u \not \equiv 0$. 
\end{proof}
\subsection{$\lambda_{1} > 0$: Mountain Pass type solution}\label{Subsection:MP}
For the case $\lambda_{1} > 0$, we use Mountain Pass theorem to obtain the weak solutions and discuss the sign of solutions.
\subsubsection{The existence of mountain pass type solution}
Similar to the arguments in Lemma~\ref{Lemma: Zk+1 inf}, it is obvious to see
\begin{Lemma}\label{Lem:MP}
	If $a(x) $ satisfies \eqref{Assumption of a} and $\lambda_{1} > 0$, then 
	\[\varsigma_{1} :=  \inf _ {\substack{u \in \mathcal{X}^{1,2}(\Omega) \\\|u\|_{\mathcal{X}^{1,2}(\Omega)}=1}} \left\{ \|u\|_{\mathcal{X}^{1,2}(\Omega)}^{2} + \int_{\Omega} a(x)u^{2} dx \right\} >0.\]
\end{Lemma}
We now obtain the Mountain Pass geometric features of $\mathcal{J}$.
\begin{Proposition}\label{Prop:MP}
	Let $\lambda_{1} > 0$. Assume that $a(x)$ satisfies \eqref{Assumption of a}, $f$ satisfies (H1)-(H3). Then,
	\begin{itemize}
		\item [$(a)$] there exist $\gamma, R>0$ such that $\mathcal{J}(u) \geq R$, if $\|u\|_{\mathcal{X}^{1,2}(\Omega)}=\gamma$.
		\item [$(b)$] there exists $e \in \mathcal{X}^{1,2}(\Omega)$ such that $\|e\|_{\mathcal{X}^{1,2}(\Omega)}>\gamma$ and $\mathcal{J}(e)<R$.
	\end{itemize}
\end{Proposition}
\begin{proof}
	The proof of part $(a)$ is obvious by  Lemma~\ref{Lem:MP}.
	We just prove part $(b)$. 
	Fix $\varphi  \in \mathcal{X}^{1,2}(\Omega)$ such that $\|\varphi \|_{\mathcal{X}^{1,2}(\Omega)}=1$.
	Let $t >0$. We have
	\begin{equation}\label{15}
		\begin{aligned}
			\mathcal{J}(t \varphi)
			&= \frac{1}{2} \|t \varphi\|_{\mathcal{X}^{1,2}(\Omega)}^{2} + \frac{1}{2} \int_{\Omega} a(x) |t \varphi|_{2}^{2} dx - \int_{\Omega} F(x,t \varphi)dx\\
			&\leqslant \frac{t^{2}}{2} \left(\| \varphi\|_{\mathcal{X}^{1,2}(\Omega)}^{2} +  |a(x)|_{\frac{n}{2}}| \varphi |_{2^{*}}^{2} \right) - \int_{\Omega} M t^{2} \varphi^{2} d x+\int_{\Omega} C_{M} d x \\
			& \leqslant t^{2} \left( \frac{1 + C |a(x)|_{\frac{n}{2}} } {2}- M|\varphi|_{2}^{2} \right) +C_{M}|\Omega|,
		\end{aligned}
	\end{equation}
	thanks to Proposition~\ref{Prop: continuous and compact embedding}  and \eqref{H3}.
	Let $M= \frac{3+C |a(x)|_{n/2}}{2|\varphi|_{2}^{2}}$. Passing to the limit as $t \rightarrow +\infty$, $\mathcal{J}(t \varphi ) \rightarrow-\infty$. 
	
	The assertion follows taking $e=T \varphi $, with $T$ sufficiently large.
\end{proof}
Now we  show the rest part of Theorem \ref{Thm: LandMP}.
\begin{proof}[\bf{Proof of Theorem \ref{Thm: LandMP} when $\lambda_{1} > 0$}]
	Since the geometry of the Mountain Pass Theorem is assured by Proposition~\ref{Prop:MP} and the $(PS)_{c}$ condition is obtained by Proposition~\ref{Prop:PSc}, we can exploit the Mountain Pass Theorem to find a critical point $v\in\mathcal{X}^{1,2}(\Omega)$ of $\mathcal{J}$. Furthermore,
	\[\mathcal{J}(v)\geqslant \inf_{\|v\|_{\mathcal{X}^{1,2}(\Omega)}=\gamma } \mathcal{J}(v) \geqslant R >0=\mathcal{J}(0),\]
	and so $v \not \equiv 0$. 
\end{proof}

\subsubsection{Some comments on the sign of the solutions}
As in the cases of the Laplacian \cite[Remark 5.19]{RP86} and fractional Laplacian \cite[Corollary 13]{SV12}, one can determine the sign of the Mountain Pass type solutions of problem~\eqref{problem}. 

\begin{Corollary}\label{Cor:sign}
	Let $\lambda_{1} > 0$, $f$ satisfy (C), (H1)-(H4). If $a(x)$ satisfying \eqref{Assumption of a} is a function with constant sign, then problem~\eqref{problem} admits both a non-negative weak solution $0 \not \equiv u_{+} \in \mathcal{X}^{1,2}(\Omega)$ and a non-positive weak solution $0 \not \equiv u_{-} \in \mathcal{X}^{1,2}(\Omega)$.
\end{Corollary}

In order to seek non-negative and non-positive solution of problem~\eqref{problem}, it is enough to introduce the following problem
\begin{equation}\label{problem2}
	\left\{%
	\begin{array}{ll}
		-\Delta u + (-\Delta)^{s} u +a(x) u^{\pm} = f^{\pm}(x,u) & \hbox{in $\Omega$,} \\
		u=0 &  \hbox{in $\mathbb{R}^n\backslash\Omega$} \\
	\end{array}%
	\right.
\end{equation}
where $u^{+}=\max \{u, 0\}, u^{-}=\min \{u, 0\}$ and
\[f^{+}(x, t)=\left\{\begin{array}{ll}
	f(x, t) & \text { if } t \geqslant 0 \\
	0 & \text { if } t<0
\end{array}, \quad  f^{-}(x, t)=\left\{\begin{array}{ll}
	0 & \text { if } t>0 \\
	f(x, t) & \text { if } t \leqslant 0
\end{array}.\right.\right.\]
The problem \eqref{problem2} has a variational structure, indeed it is the Euler-Lagrange equation of the functional $\mathcal{J}^{\pm}: \mathcal{X}^{1,2}(\Omega) \rightarrow \mathbb{R} $ defined as follows
\[\mathcal{J}^{\pm}= \frac{1}{2}\|u\|_{\mathcal{X}^{1,2}(\Omega)}^{2} + \frac{1}{2} \int_{\Omega} a(x) (u^{\pm})^2 - \int_{\Omega} F^{\pm}(x,u)dx\]
where
$F^{ \pm}(x, t)=\int_{0}^{t} f^{ \pm}(x, \tau) d \tau.$
It is easy to see $\mathcal{J}^{\pm}$ is Fr\'{e}chet differentiable in $u \in \mathcal{X}^{1,2}(\Omega)$ and for any $\phi \in \mathcal{X}^{1,2}(\Omega)$
\begin{equation}\label{17}
	\langle \nabla \mathcal{J}^{\pm}(u), \phi \rangle= B_{s} ( u, \phi ) + ( a(x) u^{\pm}, \phi)_{L^{2}(\Omega )} -  \int_{\Omega} f^{\pm}(x, u) \phi dx. 
\end{equation}

In order to prove Corollary~\ref{Cor:sign}, we only need to find a non-trivial critical point $u_{+} \geqslant 0$ (or $u_{-} \leqslant 0$) a.e. in $\mathbb{R}^{n}$ of $\mathcal{J}^{+}$ (or $\mathcal{J}^{-}$). In fact, if $u_{+}$ is a critical point of $\mathcal{J}^{+}$, then $u_{+}$ is a weak solution of problem~\eqref{problem2}. If we have in addition that $u_{+} \geqslant 0$ a.e. in $\mathbb{R}^{n}$, then $\mathcal{J}^{+}(u_+)=\mathcal{J}(u_+)$ and $u_{+}$ is also a weak solution of problem~\eqref{problem}.

\begin{proof}[\bf{Proof of Corollary~\ref{Cor:sign}}]
	Since $f$  satisfies (C), (H1)-(H4), we know $f^{+}$ satisfies (C), (H1), (H2) and 
	\begin{itemize}
		\item[\textbf{(H3')}] $\lim \limits_{t \rightarrow +\infty} \frac{F^{+}(x, t)}{t^{2}}=+\infty$ uniformly for a.e. $x \in \Omega$;
		\item[\textbf{(H4')}] there exists $T_{0}>0$ such that for any $x \in \Omega$, the function
		\begin{center}
			$t \mapsto \frac{f^{+}(x, t)}{t}$ is increasing in $t > T_{0}$.
		\end{center}
	\end{itemize} 
	
	As in Proposition~\ref{Prop:MP}, we can obtain the Mountain Pass geometric structure of $\mathcal{J}^{+}$. We remark that, since $a(x)$ has an invariant sign and $u^{+} \leqslant |u|$, we use 
	\[\|u\|_{\mathcal{X}^{1,2}(\Omega)}^{2} +\int_{\Omega} a(x)(u^{+})^{2} \geqslant 
	\min\{1, \varsigma_{1} \} \|u\|_{\mathcal{X}^{1,2}(\Omega)}^{2} >0\]
	to deduce estimates in \eqref{important estimates}.
	And we choose $\varphi>0$ in \eqref{15} ( Since $|\varphi (x)-\varphi(y)|^{2} \geqslant |\varphi^{+}(x)-\varphi^{+}(y)|^{2}$, we can always find $0<\varphi \in \mathcal{X}^{1,2}(\Omega)$ ).
	The $(PS)_{c}$ condition of $\mathcal{J}^{+}$ is obtained by Proposition~\ref{Prop:PSc}.
	Applying Mountain Pass Theorem, we get a non-trivial critical point $u_{+} $ of $\mathcal{J}^{+}$. So that $u_{+}$ is a weak solution of \eqref{problem2}.
	
	We now prove $u_{+} \geqslant 0$  a.e. in $\mathbb{R}^{n}$.
	Taking $\phi= u_{+}^{-} $ in \eqref{17}, we have 
	\[\begin{aligned}
		0 & =\left\langle\nabla \mathcal{J}^{+}(u_{+}), u_{+}^{-}\right\rangle \\ & =\int_{\Omega}\nabla u_{+} \cdot \nabla u_{+}^{-}+\int_{\mathbb{R}^{2 n}} \frac{(u_{+}(x)-u_{+}(y))\left(u_{+}^{-}(x)-u_{+}^{-}(y)\right)}{|x-y|^{n+2 s}} d x d y-\int_{\Omega} f_{a}^{+}(x, u_{+}) u_{+}^{-} d x \\ & =\int_{\Omega}\left|\nabla u_{+}^{-}\right|^{2} d x+\int_{\mathbb{R}^{2 n}} \frac{(u_{+}(x)-u_{+}(y))\left(u_{+}^{-}(x)-u_{+}^{-}(y)\right)}{|x-y|^{n+2 s}} d x d y \\ & =\left\|u_{+}^{-}\right\|_{ \mathcal{X}^{1,2}(\Omega )}^{2}-\int_{\mathbb{R}^{2 n}} \frac{u_{+}^{+}(x) u_{+}^{-}(y)+u_{+}^{+}(y) u_{+}^{-}(x)}{|x-y|^{n+2 s}} d x d y \geqslant\left\|u_{+}^{-}\right\|_{ \mathcal{X}^{1,2}(\Omega )}^{2} .
	\end{aligned}\]
	So, $u_{+} \geqslant 0$ a.e. in $\Omega$. Thus, $u_{+}$ is also a weak solution of \eqref{problem} and $\mathcal{J}(u_{+})=\mathcal{J}^{+}(u_{+})$.
	
	Similarly, we can obtain a non-positive weak solution $0 \not \equiv u_{-} \in \mathcal{X}^{1,2}(\Omega)$.
\end{proof}

\subsection{Infinitely many solutions under symmetry condition}\label{Sec:Infinitely many solutions under symmetry condition} \setcounter{equation}{0}
As is well known, Fountain Theorem \cite{Bartsch93} provides the existence of an unbounded sequence of critical value for a $C^1$ invariant functional. In this subsection, we apply Fountain Theorem to  obtain infinitely many weak solutions of problem~\eqref{problem}.

Choosing $G:=\mathbb{Z}/2=\{1,-1\}$ as the action group on $\mathcal{X}^{1,2}(\Omega)$, $X_{j}:=\mathbb{R}e_{j}$ where $\{e_{j}\}_{j\in \mathbb{N}}$ is defined as eigenfunctions in Proposition~\ref{Prop:eigen} and $V:=\mathbb{R}$, it is easy to see that $\mathcal{X}^{1,2}(\Omega)$ 
 satisfies the following conditions:
there is a compact group $G$ acting isometrically on  $\mathcal{X}^{1,2}(\Omega)= \overline{\oplus_{j \in \mathbb{N}} X_{j}}$, the spaces $X_{j}$ are invariant and there exists a finite dimensional space $V$ such that, for every $j \in \mathbb{N}, X_{j} \simeq V$ and the action of $G$ on $V$ is admissible.

Here we use  Borsuk-Ulam Theorem \cite{Borsuk1933DreiS} to prove $G$ is admissible on $\mathbb{ R }$.
While, by (S), $\mathcal{J}$ is an invariant functional for any action $g\in G$. 
And the $(PS)_{c}$ condition is obtained by Proposition~\ref{Prop:PSc}.
Now we just need to verify the functional $\mathcal{J}$ satisfies Fountain geometric structures:
\begin{itemize}
	\item [(FG)] for every $k\in \mathbb{N}$, there exists $\rho_{k}>\gamma_{k}>0$ such that
	\begin{itemize}
		\item [(i)] $a_{k}:=\max \left\{\mathcal{J}(u): u \in Y_{k},\|u\|_{\mathcal{X}^{1,2}(\Omega)}=\rho_{k}\right\} \leqslant 0$,
		\item [(ii)]  $b_{k}:=\inf \left\{\mathcal{J}(u): u \in Z_{k},\|u\|_{\mathcal{X}^{1,2}(\Omega)}=\gamma_{k}\right\} \rightarrow +\infty$, $k \rightarrow +\infty$.
	\end{itemize}
\end{itemize} 
where $Y_{k}, Z_{k}$ are defined in \eqref{space:YZ}.
We first give the following lemma.

\begin{Lemma}\label{Lamma Fountain}
	Let $1 \leqslant q<2^{*}$ and, for any $k \in \mathbb{N}$, let
	\[\beta_{k}:=\sup \left\{\|u\|_{L^q(\Omega)}: u \in Z_{k},\|u\|_{\mathcal{X}^{1,2}(\Omega)}=1\right\} .\]
	Then, $\beta_{k} \rightarrow 0$ as $k \rightarrow \infty$.
\end{Lemma}
\begin{proof}
	Since $Z_{k+1} \subset Z_{k}$, $ \beta_{k}>0$ is nonincreasing. Hence, there exist $\beta \in \mathbb{R}$ such that
	$\beta_{k} \rightarrow \beta\geqslant 0, \quad k \rightarrow+\infty.$
	Moreover, by definition of $\beta_{k}$, for any $k \in \mathbb{N}$ there exists $u_{k} \in Z_{k}$ such that
	\begin{equation}\label{11}
		\left\|u_{k}\right\|_{\mathcal{X}^{1,2}(\Omega)}=1 \text { and }\left\|u_{k}\right\|_{L^q(\Omega)}>\beta_{k} / 2.
	\end{equation}
	Since $\mathcal{X}^{1,2}(\Omega)$ is a Hilbert space, there exist $u_{\infty} \in \mathcal{X}^{1,2}(\Omega)$ and a subsequence of $u_{k}$ (still denoted by $u_{k}$ ) such that $u_{k} \rightharpoonup u_{\infty}$ in  $\mathcal{X}^{1,2}(\Omega)$. Since each $Z_{k}$ is convex and closed, hence it is closed for the weak topology. Consequently,
	$u_{\infty} \in \cap_{k=1}^{+\infty} Z_{k}=\{0\}$.
	By Proposition~\ref{Prop: continuous and compact embedding}, we get
	$u_{k} \rightarrow 0 $ in $ L^{q}(\Omega)$. Together with \eqref{11}
	we get that $\beta_{k} \rightarrow 0$ as $k \rightarrow+\infty$.
\end{proof}
\begin{proof}[\bf {Proof of Theorem~\ref{Thm: infinitely many sol Theorem}}]
We just prove that $\mathcal{J}$ has Fountain geometric feature (FG).

 Firstly, we verify the assumption (ii) . Since $f$ satisfies (H1), there exists a constant $C>0$ such that
	\begin{equation*}
		\begin{aligned}
			|F(x,u)| \leqslant \int_{0}^{u} \left|f(x,s)\right| ds \leqslant \int_{0}^{u} C_{f}(1+|s|^{q-1}) ds 
			\leqslant C(1+|u|^{q})
		\end{aligned}
	\end{equation*}
	for a.e. $x \in \bar{\Omega}$ and $u \in \mathbb{R}$. 
	
	Take any $k \in \mathbb{N}$. Then, for any $u \in Z_{k} \backslash\{0\}$, by Lemma~\ref{Lem:MP}, we obtain
	\[\begin{aligned}
		\mathcal{J}(u) &
		\geqslant \frac{\varsigma_{1}}{2}\|u\|_{\mathcal{X}^{1,2}(\Omega)}^{2}-C|u|_{q}^{q}-C|\Omega| \\
		& =\frac{\varsigma_{1}}{2}\|u\|_{\mathcal{X}^{1,2}(\Omega)}^{2}-C\left|\frac{u}{\|u\|_{\mathcal{X}^{1,2}(\Omega)}}\right|_{q}^{q}\|u\|_{\mathcal{X}^{1,2}(\Omega)}^{q}-C|\Omega| \\
		&\geqslant \frac{\varsigma_{1}}{2}\|u\|_{\mathcal{X}^{1,2}(\Omega)}^{2}-C \beta_{k}^{q}\|u\|_{\mathcal{X}^{1,2}(\Omega)}^{q}-C|\Omega|\\
		&=\|u\|_{\mathcal{X}^{1,2}(\Omega)}^{2}\left(\frac{\varsigma_{1}}{2}-C \beta_{k}^{q}\|u\|_{\mathcal{X}^{1,2}(\Omega)}^{q-2}\right)-C|\Omega|
	\end{aligned}\]
	where $\beta_{k}$ is defined as in Lemma \ref{Lamma Fountain} . Choosing
	\[\gamma_{k}=\left( \frac{q}{ \varsigma_{1} } C \beta_{k}^{q}\right)^{-1 /(q-2)},\]
	it is easy to see that $\gamma_{k} \rightarrow+\infty$ as $k \rightarrow+\infty$, thanks to Lemma~\ref{Lamma Fountain} and the fact that $q>2$. As a consequence, we get that for any $u \in Z_{k}$ with $\|u\|_{\mathcal{X}^{1,2}(\Omega)}=\gamma_{k}$,
	\[\mathcal{J}(u) \geqslant \varsigma_{1} \left(\frac{1}{2}-\frac{1}{q}\right) \gamma_{k}^{2}-C|\Omega| \rightarrow+\infty\]
	as $k \rightarrow+\infty$.
	
It remains to verify the assumption (i). Since, on the finite dimensional space $Y_{k}$ all norms are equivalent, by \eqref{finite dim eq norm}, \eqref{H3} and Proposition~\ref{Prop: continuous and compact embedding}, we have, for any $u\in Y_{k}$
	\[\begin{aligned}
		\mathcal{J}(u) &\leqslant \frac{1}{2} \left( \|u\|^{2}_{\mathcal{X}^{1,2}(\Omega)} + |a|_{\frac{n}{2}}|u|_{2^{*}}^{2} \right) - M|u|_{2}^{2}+ C_{M} |\Omega| \\
		&\leqslant \frac{1}{2} \|u\|_{\mathcal{X}^{1,2}(\Omega)}^{2} \left( 1 + \tilde{C}_{k, 2^{*}}^{2}|a|_{\frac{n}{2}} -M C_{k, 2^{*}}^{2} \right) + C_{M} |\Omega|.
	\end{aligned}\]
Take $M$ and $\|u\|_{\mathcal{X}^{1,2}(\Omega)}=\rho_{k}>\gamma_{k}>0$ large enough. Then $\mathcal{J}(u) \leqslant 0$, due to the fact $\Omega$ is bounded.

In conclusion, $\mathcal{J}$ has infinitely many critical points $\{u_{j}\}_{j\in \mathbb{N}}$ and $\mathcal{J}(u_{j}) \rightarrow +\infty$ as $j\rightarrow \infty$ applying Fountain Theorem. 		
\end{proof}
\section{Regularity of weak solutions}\label{Sec: Regularity of weak solution }
\setcounter{equation}{0}
In this section, we discuss the regularity theory of weak solution to problem~\eqref{problem}. We first prove the global boundedness of weak solutions. Because the embedding \eqref{imbedding when n=1,2} is continuous for $n=1$ or $2$, it suffices to deal with the case $n \geqslant 3$. 
\subsection{Global boundedness} 
We first prove the $L^{\infty}$-regularity, see Theorem~\ref{Thm: boundedness}, of weak solutions to  problem~\eqref{problem} with the term $-a(x)u + f(x,u)$.
\subsubsection{$L^{\infty}$-regularity for $-a(x)u + f(x,u)$}
The method we take is \textit{Moser iteration} (see for example, \cite{HLbook,DMVbook}), which is based on the following fact:
if there exists a constant M (independent of $p$), such that $ |u|_{p} \leqslant M$ for a sequence $p \rightarrow \infty$, then $u \in L^{\infty}(\Omega)$.
Inspired by this , for given $\beta >1, T>0$, we define an auxiliary function $\varphi (t) : \mathbb{ R } \rightarrow \mathbb{ R }^{+}_{0}$ as 
\begin{equation}\label{29}
	\varphi ( t ) = \left\{ \begin{array} { l l } - \beta T ^ { \beta - 1 } ( t + T ) + T ^ { \beta } , & \text{if } t \leqslant - T, \\ | t | ^ { \beta } , & \text{if } - T < t < T ,\\ \beta T ^ { \beta - 1 } ( t - T ) + T ^ { \beta } , & \text{if } t \geqslant T . \end{array} \right. 
\end{equation}

Note $\varphi$ is convex. Suppose $u \in \mathcal{X}^{1,2}(\Omega)$. It is easy to check $\varphi(u) \varphi^{'}(u) \in \mathcal{X}^{1,2}(\Omega)$. Then, $\varphi(u) \varphi^{'}(u)$ can be a test function and $\int_{\Omega} a u \varphi(u) \varphi^{'}(u)dx < \infty$ is well posed.

\begin{proof}[\bf{Proof of Theorem~\ref{Thm: boundedness}}]
	We first prove the theorem under condition (1). 
	 Since $u$ is a weak solution, testing equation~\eqref{problem} for $u$ by $\varphi(u) \varphi^{'}(u)$, we obtain
	 \begin{equation}\label{21}
	 	\begin{aligned}
	 		&\quad \int_{ \mathbb { R } ^ { n } }  \nabla u \cdot \nabla (\varphi(u) \varphi^{'}(u)) ~ d x + \int_{\mathbb{R}^{n}}  \varphi(u) \varphi^{'}(u) (-\Delta)^{s}u ~d x  \\
	 		&=  \int_{\Omega} (-a(x) u + f(x,u))\varphi(u) \varphi^{'}(u)~ dx
	 	\end{aligned}
	 \end{equation}
	
	By the convexity of $\varphi$ and the definition of $(-\Delta)^{s}$, we have  $(-\Delta)^{s} \varphi (u) \leqslant \varphi^{'}(u) (-\Delta)^{s}u $. Using \textit{fractional Green's formula}, we obtain  
	\begin{equation}\label{23}
		\int_{\mathbb{R}^{n}}  \varphi(u) \varphi^{'}(u) (-\Delta)^{s}u ~d x \geqslant \int_{\mathbb{R}^{n}}  \varphi(u) (-\Delta)^{s} \varphi (u) ~d x
		=\left\| (- \Delta )^{\frac{s}{2}} \varphi(u) \right\| _ {L^{2}\left(\mathbb{R}^{n}\right)}^{2}
		\geqslant 0.
	\end{equation}
		
	Since $\varphi(u),~\varphi^{''}(u) \geqslant 0$, we have
	\begin{equation*}\label{24}
		\int_{ \mathbb { R } ^ { n } }  \nabla u \cdot \nabla (\varphi(u) \varphi^{'}(u)) ~ d x \geqslant \int_{ \mathbb { R } ^ { n } }  |\nabla u|^{2} |\varphi^{'}(u)|^{2} ~ d x.
	\end{equation*}
	Together with \eqref{21}-\eqref{23}, we obtain 
	\begin{equation}\label{27}
	\int_{ \mathbb { R } ^ { n } }  |\nabla u|^{2} |\varphi^{'}(u)|^{2} ~ d x \leqslant \int_{ \Omega } (-a(x)u + f(x,u)) \varphi(u) \varphi^{'}(u) ~ dx.
	\end{equation}
	
	Notice $a(x),~\varphi(u),~ u\varphi^{'}(u) \geqslant 0$,
	\begin{equation}\label{22}
		\int_{\Omega} -a(x) u \varphi(u) \varphi^{'}(u) ~dx \leqslant 0.
	\end{equation}
    Thus, by Sobolev-Poincar\'{e} inequality, 
    \begin{equation}\label{25}
    	\begin{aligned}
    		&\quad \| \varphi (u)\|_{L^{2^{*}}(\Omega)}^{2} \leqslant C(n, \Omega) \| \nabla \varphi (u) \|_{L^{2}(\Omega)} ^{2} = C(n, \Omega) \int_{\Omega} | \nabla u ~\varphi ^{'} (u)  |^{2} ~ dx.
    	\end{aligned}
    \end{equation}
    
	Noticing $| \varphi^{'}(u)| \leqslant \beta  |u|^{\beta -1}$
	, $|u \varphi^{'}(u)| \leqslant \beta  \varphi(u)$ 
	and  $\varphi(u) \leqslant |u|^{\beta}$, we obtain by (H1)
	\begin{equation}
		\begin{aligned}
			 \int_{ \Omega }  f(x,u) \varphi(u) \varphi^{'}(u) ~ dx
			\leqslant c_{f} \beta \int_{ \Omega } |u|^{2 \beta -1}  + (\varphi(u))^{2} |u|^{q-2} ~ dx.
		\end{aligned}
	\end{equation}

\textbf{Step 1.}
In this step, we are devoted to finding the initial state of the Moser iteration. We claim that $u \in L^{2^{*}\beta_{1} } (\Omega)$ where $\beta_{ 1 } = \frac{2^{*}+1}{2}$.

In fact, for the fixed $\beta_{1}>1$, we have
\begin{equation}\label{28}
	\begin{aligned}
		& \quad \int_{\Omega} (\varphi(u))^{2} |u|^{q-2} ~ dx
		\\& \leqslant \int_{\Omega \cap \{|u|\leqslant R\}} (\varphi(u))^{2} |u|^{q-2} ~ dx + \int_{\Omega \cap \{|u|> R\}} (\varphi(u))^{2} |u|^{q-2} ~ dx
		\\& \leqslant \int_{\Omega \cap \{|u|\leqslant R\}} \frac{(\varphi(u))^{2}}{|u|} R^{q-1 }~ dx + \left(\int_{\Omega} (\varphi(u))^{2^{*}} ~ dx \right)^{\frac{2}{2^{*}}}  \int_{ \{|u|> R\}} \left( |u|^{\frac{ 2^{*}(q-2)} {2^{*}-2}} ~ dx \right)^{\frac{2^{*}-2}{2^{*}}}
	\end{aligned}
\end{equation}
in which we can choose $R$ large enough such that
\[ \int_{ \{|u|> R\}} \left( |u|^{\frac{ 2^{*}(q-2)} {2^{*}-2}} ~ dx \right)^{\frac{2^{*}-2}{2^{*}}} \leqslant \frac{1}{2C(n,\Omega) c_{f} \beta_{1}}. \]

Combining \eqref{27}-\eqref{28}, we have 
\begin{equation}\label{35}
	\frac{1}{2}\| \varphi (u)\|_{L^{2^{*}}(\Omega)}^{2} \leqslant
	C \beta_{1} \left( \int_{ \Omega } |u|^{2 \beta_{1} -1}~dx + \int _{\Omega} |u|^{2 \beta_{1} -1} R^{q-1} ~dx \right)
\end{equation}
where $C=C(n,\Omega, c_{f})$.

This implies $u \in L^{2^{*}\beta_{1} } (\Omega)$ where $ \beta_{1}= \frac{2^{*}+1}{2}$, if $T\rightarrow \infty$ in the definition of $\varphi$.

\textbf{Step 2.}
In this step, we set up the iterative formula.

We first claim that 
\begin{equation}\label{31}
	\left( 1+ \int_{\Omega} |u|^{2^{*} \beta} ~ dx\right)^{\frac{1}{2^{*} (\beta -1 )}} \leqslant
	(C\beta)^{\frac{1}{2(\beta-1)}} \left( 1+ \int_{\Omega} |u|^{2\beta + 2^{*} -2} ~ dx \right)^{ \frac{1}{2(\beta -1)}}.
\end{equation}

In fact, by
     \begin{equation}
     	\begin{aligned}
     		&\quad \int_{ \Omega }  f(x,u) \varphi(u) \varphi^{'}(u) ~ dx
     		\leqslant \int_{ \Omega }  c_{f}~ \beta (|u|^{2\beta -1}+ |u|^{2 \beta +q -2})   ~ dx\\
     		&\leqslant  c_{f}~ \beta \left(\int_{ \Omega }  |u|^{2\beta -1} ~dx + \left|\Omega \cap \left\{|u| \leqslant 1 \right\}\right| +  \int_{ \Omega \cap \left\{|u| > 1 \right\} } |u|^{2 \beta +2^{*} -2}~ dx\right)   \\
     		&\leqslant  C~ \beta \left(1+  \int_{ \Omega \cap \left\{|u| > 1 \right\} } |u|^{2 \beta +2^{*} -2}~ dx\right),
     	\end{aligned}
     \end{equation}   where $C= C(n, \Omega, c_{f})$.
	The last inequality follows from
	\begin{equation*}
		\begin{aligned}
			\int_{\Omega} |u|^{2\beta -1} ~dx \leqslant \left( \int_{\Omega} |u|^{2\beta + 2^{*} - 2} ~dx \right)^{\frac{2\beta -1}{2\beta + 2^{*} - 2}} ~ |\Omega|^{\frac{2^{*}-1}{2\beta + 2^{*} - 2}}
			\leqslant \int_{\Omega} |u|^{2\beta + 2^{*} - 2} ~dx + |\Omega|
		\end{aligned}		
	\end{equation*}   (due to H\"{o}lder inequality and Young inequality).
   
    Together with \eqref{27}-\eqref{25}, we have 
   \begin{equation}\label{30}
   	\| \varphi (u)\|_{L^{2^{*}}(\Omega)}^{2}
   	\leqslant 
   	C \beta \left(1+  \int_{ \Omega } |u|^{2 \beta +2^{*} -2}~ dx\right).
   \end{equation}
  
   Therefore, let $T \rightarrow \infty$, we have
\begin{equation*}
	\left(1+ \int_{\Omega} |u|^{2^{*} \beta} ~ dx\right)^{2} \leqslant
	2+2( C\beta)^{2^{*}} \left( 1+ \int_{\Omega} |u|^{2\beta + 2^{*} -2} ~ dx \right)^{2^{*}},
\end{equation*}
using $(a+b)^{2} \leqslant 2(a^{2}+ b^{2})$. Taking the $2\cdot 2^{*}  (\beta -1)$-th root, we have \eqref{31}.

Define the parameters $(\beta_m)_{m\in \mathbb{Z}_+}$ iteratively by 
\begin{equation}\label{33}
	 2 \beta_{m+1} + 2^{*} -2 = 2^{*} \beta_{m}.
\end{equation}
Thus, $2 (\beta_{m+1} -1) = 2^{*} (\beta_{m}-1)$. 

Taking $\beta= \beta_{m+1}$ in \eqref{31}, we have the iterative formula
\begin{equation}\label{36}
	\left( 1+ \int_{\Omega} |u|^{2^{*} \beta_{m+1}} ~ dx\right)^{\frac{1}{2^{*} (\beta_{m+1} -1 )}} \leqslant
	C_{m+1}~^{\frac{1}{2(\beta_{m+1}-1)}} \left( 1+ \int_{\Omega} |u|^{ 2^{*} \beta_{m} } ~ dx \right)^{ \frac{1}{2^{*}(\beta_{m} -1)}}
\end{equation}
where $C_{m+1}=C\beta_{m+1}$.

\textbf{Step 3.}
We deduce that, for every $\beta_{m+1}$, $u \in L^{2^{*} \beta_{m+1}}(\Omega )$.

In fact, by performing $m$-th iterations, we have
\[\left( 1+ \int_{\Omega} |u|^{2^{*} \beta_{m+1}} ~ dx\right)^{\frac{1}{2^{*} (\beta_{m+1} -1 )}} \leqslant
\prod \limits _{k=2}^{m+1} C_{k}~^{\frac{1}{2(\beta_{k}-1)}} \left( 1+ \int_{\Omega} |u|^{ 2^{*} \beta_{1} } ~ dx \right)^{ \frac{1}{2^{*}(\beta_{1} -1)}}.\]

Now we turn to prove
\begin{equation}\label{32}
	\prod \limits _{k=2}^{m+1} C_{k}~^{\frac{1}{2(\beta_{k}-1)}} \leqslant C_{0}.
\end{equation} 

Denote $\bar{q}:= \frac{2}{2^{*}} < 1$. Since
\[ \beta_{m+1} = \left( \frac{2^{*}}{2} \right)^{m} (\beta_{1}-1)+1 =  \left( \frac{2^{*}}{2} \right)^{m +1}-  \frac{1}{2} \left( \frac{2^{*}}{2} \right)^{m } +1  \leqslant 2 \bar{q}~^{-(m+1)}, \]
we have
$C_{k}= C\beta_{k} \leqslant 2C \bar{q}~^{-k} $. We still denote $2C$ by $C$. 

Thus, by \eqref{33}, we have
\begin{equation*}
	\begin{aligned}
		\prod \limits _{k=2}^{m+1} C_{k}~^{\frac{1}{2(\beta_{k}-1)}}
		\leqslant \prod \limits _{k=2}^{m+1} \left( C \bar{q}~^{-k} \right)^{\frac{1}{2(\beta_{k}-1)}}
		=  \prod \limits _{k=2}^{m+1} \left( C \bar{q}~^{-k} \right)^{\frac{ \bar{q}^{k-1}}{2(\beta_{1}-1)}}
		= \left(C ^{ \sum \limits _{k=2}^{m+1} \bar{q}^{k-1} } \cdot \bar{q}^{ - \sum \limits _{k=2}^{m+1} k \cdot \bar{q}~^{k-1} } \right)^{\frac{1}{2(\beta_{1}-1)}}.
	\end{aligned}
\end{equation*}
Since
$ \sum \limits _{k=2}^{m+1} \bar{q}^{k-1} $ is a geometric sequence with common ratio $ \bar{q} <1$ and 
\[\sum \limits _{k=2}^{m+1} k \cdot \bar{q}~^{k-1}=\left( \sum \limits _{k=2}^{m+1}  \bar{q}~^{k} \right)^{'} \]
is a power series with base number  $ \bar{q} <1$, we have \eqref{32}.

Let $m\rightarrow \infty$. We proved Theorem~\ref{Thm: boundedness} under assumption (1).

The proof under condition (2) is quite similar to that  under condition (1). The only differences are:
\begin{itemize}
	\item [(i)] the estimate \eqref{22} is substituted by 
	\[\int_{\Omega} -a(x) u \varphi(u) \varphi^{'}(u) ~dx \leqslant \beta \| a(x) \|_{L^{\infty}(\Omega)} \int_{\Omega} |u|^{2 \beta} dx, \]
	\item [(ii)] the inequality \eqref{35} is substituted by
	$$ \| \varphi (u)\|_{L^{2^{*}}(\Omega)}^{2} \leqslant
	C \beta_{1}  R^{q-2} \left( 1 + \int_{ \Omega } |u|^{2 \beta_{1} }~dx  \right)$$
	 where  $ \beta_{1}= \frac{2^{*}}{2}, C=C_0(n, \Omega, c_{f}, |a|_{\infty} )$.
\end{itemize} 
Thus, we finish the proof of Theorem~\ref{Thm: boundedness}.
\end{proof}

We now  deduce the boundedness of variational weak solutions obtained in section~\ref{Sec: Global existence results}.
\begin{Corollary}\label{existence and boundedness}
	Assume  $f(x,u)$ satisfies (C), (H1)-(H4).
	\begin{itemize}
		\item [(i)] If $0< a(x) \in L^{\frac{l}{2}}(\Omega) ,l \geqslant n $, then problem~\eqref{problem} admits a bounded Mountain Pass type weak solution. If we further assume $f$ satisfies (S), then problem~\eqref{problem} admits infinitely many bounded Fountain type weak solution.
		\item [(ii)] If $ a(x) \in L^{\infty}(\Omega)$, then problem~\eqref{problem} admits a bounded Mountain Pass (or Linking) weak solution for $\lambda_{1} >0$ (or $\lambda_{1} \leqslant 0$).
	\end{itemize}
\end{Corollary}
\begin{remark}
	When the nonlinearity $f(x,u)$ is critical, we can still deduce that a weak solutions of problem~\eqref{problem} is bounded. However, the global existence of weak solutions is hard to prove. (See \cite[Theorem 1.3 and Theorem 1.4]{SSEE22}.) 
\end{remark}

\subsubsection{$L^{\infty}$-regularity for $-a(x)u + f(x)$}We now prove an $L^{\infty}$-regularity--Theorem~\ref{Thm: boundedness2}, of weak solutions to problem~\eqref{problem} when $f= f(x)$
and  $a(x)$ is not always non-negative.
Noticing the Moser iterative formula is no longer applicable, we use \textit{De Giorgi iteration} (see for example, \cite{HLbook, GTbook}) to prove this theorem.
\begin{proof}[\bf{Proof of Theorem~\ref{Thm: boundedness2}}]
Let $k>0$. Consider $A_{k}= \{ u>k \}$. Set $v=(u-k)^{+} \in \mathcal { X } ^ { 1 , 2 } ( \Omega )$ as the test function. Note $v=u-k, Dv=Du$ a.e. in $A_{k}$ and $v=0, Dv=0$ a.e. in $\{ u \leqslant k \}$.
    
Since $u$ is a weak solution, we have
    \begin{equation*}
    	\int_{ A_{k}}  | D v |^{2} ~ d x + \int_{\mathbb{R}^{2n}} \frac { (u ( x ) - u ( y ) )( v ( x ) - v ( y ) ) } { | x - y | ^ { n + 2 s } } d x d y = \int_{A_{k}} (-au + f) v ~dx.
    \end{equation*}

By some simple calculation, we obtain, for any $ \phi \in \mathcal { X } ^ { 1 , 2 } ( \Omega )$
\[ (\phi ( x ) - \phi ( y ) )( \phi^{+} ( x ) - \phi^{+} ( y )  ) \geqslant |\phi^{+} ( x ) - \phi^{+} ( y )|^{2}, \quad \forall x,y \in \mathbb{R}^{n}.\]
Taking $\phi= u-k $, we have 
$(u ( x ) - u ( y ) )( v ( x ) - v ( y ) ) \geqslant |v(x) - v(y)|^{2}.$
Therefore, 
\begin{equation}\label{37}
	\int_{ A_{k}}  | D v |^{2} ~ d x  \leqslant \int_{A_{k}} (-au + f) v ~dx.
\end{equation}

Since $ gh \leqslant g^{2}+ h^{2}$ for $g>0,h >0$, we have
\begin{equation}\label{38}
	\begin{aligned}
		& \quad \int_{A_{k}} -auv ~dx= \int_{A_{k}} -a(v+k)v ~dx \leqslant  2\int_{A_{k}} |a|(v^{2} + k^{2}) ~dx
		\\& \leqslant 2 \left(|a|_{l} \left(\int_{A_{k}} v^{2^{*}}~dx \right)^{\frac{2}{2^{*}}} |A_{k}|^{1- \frac{2}{2^{*}} -\frac{1}{l} } + k^{2} |a|_{l}~|A(k)|^{1-\frac{1}{l}} \right)
		\\& \leqslant C \left( |a|_{l} \int_{A_{k}} |Dv|^{2} dx ~|A_{k}|^{\frac{2}{n} - \frac{1}{l} } +  k^{2} |a|_{l}~|A(k)|^{1-\frac{1}{l}} \right).
	\end{aligned}
\end{equation}

Using H\"{o}lder inequality and \textit{Young inequality with $\delta$}, we have 
\begin{equation}\label{39}
	\begin{aligned}
		 \int_{A_{k}} f(x) v ~dx &\leqslant |f|_{l} ~|v|_{2^{*}} |A(k)|^{ 1- \frac{1}{2^{*}} -\frac{1}{l} }
		\\& \leqslant C~|f|_{l} \|Dv\|_{L^{2}(A_{k})} |A_{k}|^{ \frac{1}{2} + \frac{1}{n} -\frac{1}{l}}
	    \\& \leqslant C \left( \delta \int_{A_{k}} |Dv|^{2} ~ dx + C_\delta |f|_{l}^{2} |A_{k}|^{1+ \frac{2}{n} -\frac{2}{l} } \right).
	\end{aligned}
\end{equation}
 Note $1+\frac{2}{n}-\frac{2}{l} > 1-\frac{1}{l}$ if $l>\frac{n}{2}$.
 Combining \eqref{37}-\eqref{39}, we have
 \begin{equation*}
 	\begin{aligned}
 		 \int_{ A_{k}}  | D v |^{2} ~ d x  
 		\leqslant C \left(  |A_{k}|^{\frac{2}{n} - \frac{1}{l} } \int_{A_{k}} |Dv|^{2}  +  k^{2} |A(k)|^{1-\frac{1}{l}} + \delta \int_{A_{k}} |Dv|^{2} + C_\delta |f|_{l}^{2} |A_{k}|^{1+ \frac{2}{n} -\frac{2}{l} } \right)
 	\end{aligned}
 \end{equation*}
 where $C=C(n, \Omega, |a|_{l})$. Since $|A_{k}|$ is decreasing with respect to $k$, there exists $k_{0}$ large enough such that \[|A_{k}| < \min \left\{1,  \left(\frac{1}{4C} \right)^{\frac{1}{\frac{2}{n} - \frac{1}{l} } } \right\} \text{ for any } k \geqslant k_{0}.\]
 Take $\delta= \frac{1}{4C}$. For every $k\geqslant k_{0}$, we have
 \begin{equation}\label{40}
 	\begin{aligned}
 		\int_{ A_{k}}  | D v |^{2} ~ d x  \leqslant C k^{2} |A(k)|^{1-\frac{1}{l}} 
 	\end{aligned}
 \end{equation}
where $C=C(n, \Omega, |a|_{l}, |f|_{l}) $.

For $\forall h>k$, we have $A(h) \subset A(k)$. Thus, $\int_{A_{h}} (u-h)^{2} \leqslant \int_{A_{k}} (u-k)^{2}$ and 
\[ |A(h)|= |\{ u-k > h-k \}| \leqslant \int_{A(h)} \frac{(u-k)^{2}}{ (h-k)^{2} } \leqslant  \frac{1}{ (h-k)^{2}} \int_{A(k)} (u-k)^{2}. \]
Note
\[\int_{A_{k}} (u-k)^{2} = \int_{A_{k}} v^{2} \leqslant C \left( \int_{A_{k}} v^{2^{*}} \right)^{\frac{2}{2^{*}}} |A(k)|^{1- \frac{2}{2^{*}}} \leqslant C \int_{A_{k}} |Dv|^{2} |A(k)|^{\frac{2}{n}}.\]
Together with \eqref{40}, we have 
\begin{equation*}\label{41}
	\int_{A_{k}} (u-k)^{2} \leqslant C  k^{2} |A(k)|^{1 + \frac{2}{n} -\frac{1}{l} }  \leqslant C  k^{2} |A(k)|^{1 + \epsilon}, \quad \epsilon < \frac{2}{n} -\frac{1}{l}
\end{equation*}
for $\forall k \geqslant k_{0}$.
Thus, for $\forall h>k \geqslant k_{0}$,  we have
\begin{equation*}
	\begin{aligned}
		 \int_{A_{h}} (u-h)^{2} &\leqslant C  h^{2} |A(h)|^{1 + \epsilon} 
		 \leqslant C  h^{2} \left( \frac{1}{ (h-k)^{2}} \int_{A(k)} (u-k)^{2} \right)^{1 + \epsilon} 
		\\& \leqslant C \frac{h^{2}}{ (h-k)^{2}} \frac{1}{ (h-k)^{2\epsilon}} \left( \int_{A(k)} (u-k)^{2} \right)^{1 + \epsilon} 
	\end{aligned}
\end{equation*}
or
\begin{equation}\label{43}
	\| (u-h)^{+} \|_{L^{2}(\Omega)} \leqslant C \frac{h}{ h-k } \frac{1}{ (h-k)^{\epsilon}} \| (u-k)^{+} \|_{L^{2}(\Omega)}^{1 + \epsilon}.
\end{equation}

Define the iterative parameters $(k_j)_{j\in \mathbb{N}}$ as 
\begin{equation*}\label{42}
	k_{j}= k_{0} + k (1- \frac{1}{2^{j}}), \quad j=0,1,2, \cdots.
\end{equation*}
Note $k_{j} \leqslant k_{0} + k$, $k_{j} - k_{j-1} = \frac{k}{2^{j}} $ and $k_{j} \rightarrow k_{0} +k$ as $j \rightarrow \infty$.

Set $\varphi (k) = \| (u-k)^{+}\|_{L^{2}(\Omega)}$. Let $h= k_{j}, k=k_{j-1}$ in \eqref{43}. We have the iterative formula
\begin{equation}\label{44}
	\begin{aligned}
		\varphi (k_{j}) 
		 \leqslant C \frac{2^{j}(k_{0}+k)}{k} \frac{2^{\epsilon j}}{k^{\epsilon}} \varphi(k_{j-1})^{1+\epsilon}.
	\end{aligned}
\end{equation}

We claim that for any $j=0,1,2,\cdots$,
\begin{equation}\label{45}
		\varphi(k_{j}) \leqslant \frac{\varphi(k_{0})}{ \nu^{j}}, \quad \text{ for some } \nu>1
\end{equation}
if $k$ is sufficiently large. 

We prove by induction. Obviously  \eqref{45} is true for $j=0$. Suppose it is true for $j-1$. Then, 
\[ \varphi(k_{j-1})^{1+\epsilon} \leqslant \left( \frac{\varphi(k_{0})}{ \nu^{j-1}} \right)^{1+\epsilon}
\leqslant \frac{\varphi(k_{0})^{\epsilon}}{ \nu^{(j-1)(1+\epsilon)-j}} \frac{\varphi(k_{0})}{ \nu^{j}}.\]
By \eqref{44}, we have
\begin{equation}
	\begin{aligned}
		\varphi(k_{j}) &\leqslant C \frac{2^{j}(k_{0}+k)}{k} \frac{2^{\epsilon j}}{k^{\epsilon}} \cdot \frac{\varphi(k_{0})^{\epsilon}}{ \nu^{(j-1)(1+\epsilon)-j}} \frac{\varphi(k_{0})}{ \nu^{j}}
		\\& = C \nu^{1+ \epsilon} \cdot \frac{k_{0}+k}{k} \cdot \left( \frac{\varphi (k_{0})}{k} \right)^{\epsilon} \cdot \frac{2^{j(1+\epsilon)}}{\nu^{j \epsilon}} \cdot \frac{\varphi(k_{0})}{ \nu^{j}}.
	\end{aligned}
\end{equation}
Take $ \nu^{\epsilon} = 2^{1+\epsilon}$. Choose $k= C_{*}(k_{0} + \varphi(k_{0}))$, for $C_{*}$ large enough. Then \eqref{45} follows from
\[ C \nu^{1+ \epsilon} \cdot \frac{k_{0}+k}{k} \cdot \left( \frac{\varphi (k_{0})}{k} \right)^{\epsilon} \cdot \frac{2^{j(1+\epsilon)}}{\nu^{j \epsilon}} \leqslant  C \nu^{1+ \epsilon} \cdot2 \cdot \left( \frac{1}{C_{*}} \right)^{\epsilon} \cdot 1 \leqslant 1. \]

Let $j\rightarrow \infty$ in \eqref{45}, then $\varphi (k_{0}+k) = 0$, i.e., 
\[ \| (u-(k_{0} + k))^{+}\|_{L^{2}(\Omega)} = 0.\]
Thus, for a.e. $x \in \Omega$,
\[ \sup_{\Omega} u \leqslant k_{0}+k \leqslant (C_{*}+1)(k_{0}+ \varphi(k_{0})) < \infty.\]
Since $-u$ is the weak solution of problem~\eqref{problem} with $-f$, we have
\[ \inf_{\Omega} u = -\sup_{\Omega} (-u) \geqslant - (C_{*}+1)(k_{0}+ \varphi(k_{0})). \]

Now we deduce that $u \in L^{\infty}(\Omega)$.
\end{proof}

\subsection{ $C^{2,\alpha}$-regularity}\label{Subsec: Up to C2-regularity}  
In this section, we derive the interior $C^{2, \alpha}$-regularity and the $C^{2, \alpha}$-regularity up to the boundary of weak solutions for the mixed operator $\mathcal{L}_a$.
	The proofs employ techniques similar to those established in ~\cite[Theorems~1.5 and 1.6]{SVWZ25}, we summarize the key steps below for completeness. Accordingly, such solutions of problem \eqref{problem} have some symmetry properties.

Before presenting the regularity results, we first introduce some notations that will be used throughout this section.
	
	\begin{enumerate}[(a)]
		\item Define $g(x,u) := -a(x)u + f(x,u)$. 
		
		\item For a bounded function $u \in \mathcal{X}^{1,2}(\Omega)$, let $I_u = [-\|u\|_{L^\infty(\Omega)}, \|u\|_{L^\infty(\Omega)}]$.
		
		\item For an open set $V$ with $V \subset \subset \Omega$, define
		\[	\rho = \mathrm{dist}(V, \partial \Omega), \quad \text{and} \quad V_\delta = \{x \in \Omega : \mathrm{dist}(x, V) < \delta\}. \]
		
		\item For any given $x_0 \in V_{\rho/4}$, set
		\[ 	0 < R < \min\left(1/2, \rho/10\right), \quad B_R(x_0) = \{x \in \Omega : |x - x_0| < R\}. \]
		
		\item Define the interior norms as
		\[ [u]_{\alpha; B_R(x_0)} = \sup_{\substack{x, y \in B_R(x_0),\, x \neq y}} \frac{|u(x) - u(y)|}{|x - y|^{\alpha}}, \quad 0 < \alpha < 1; \]
		\[	|u|'_{k; B_R(x_0)} = \sum_{j=0}^{k} R^j \|D^ju\|_{L^\infty(B_R(x_0))}; \quad
		|u|'_{k, \alpha; B_R(x_0)} = |u|'_{k; B_R(x_0)} + R^{k+\alpha}[D^ku]_{\alpha; B_R(x_0)}. \]
	\end{enumerate}

\begin{Theorem}[Interior $C^{1, \alpha}$-regularity]\label{Thm: Interior C1-regularity}
	Suppose $u \in \mathcal { X } ^ { 1 , 2 } ( \Omega )$ is a bounded weak solution of 
	\[-\Delta u + (-\Delta)^{s} u+ a(x) u =f(x,u)  \quad \text{in }\Omega, \]
	where $a(x) \in L^{\infty}(\Omega)$ and $f(x,t) \in L^{\infty}_{loc}(\Omega \times \mathbb{R})$. Assume $V$ is an open domain with $ V \subset \subset \Omega$. Then, $u \in C^{1, \alpha}(\bar{V})$ for any $\alpha \in (0, 1)$.
\end{Theorem}
\begin{proof}
		The proof follows via a truncation method and a covering argument as in \cite[Theorem 1.4]{SVWZ25}. For the reader's convenience, we provide a sketch of the proof.
				
		\textbf{Step 1.} Regularize the solution by the standard mollifier.
		For \( 0 < \varepsilon < R \), define the mollification
		\[ 
		u_\varepsilon(x) = (\eta_\varepsilon * u)(x) = \int_\Omega \eta_\varepsilon(x-y)u(y)\,dy,
		\]
		which satisfies
		\[
		-\Delta u_\varepsilon + (-\Delta)^s u_\varepsilon = g_\varepsilon \quad \text{in } V_{3\rho/4},
		\]
		where \( g_\varepsilon = \eta_\varepsilon * g(x,u) \). 
		
		Recalling the assumptions on \( a \) and \( f \), since \( u \) is bounded, it follows that \( g(x,u) \) belongs to \( L^\infty_{\mathrm{loc}}(\Omega \times \mathbb{R}) \). Moreover, the standard properties of convolution imply that:
		\begin{itemize}
			\item \( u_\varepsilon \in C^2(\overline{V}_{3\rho/4}) \cap L^\infty(\mathbb{R}^n) \) and 
			 $$\|u_\varepsilon\|_{L^\infty(\mathbb{R}^n)} \leqslant \|u\|_{L^\infty(\mathbb{R}^n)} .$$
			\item For every $ x_0 \in V_{\rho/4}$,  one has that
			$$ \|g_\varepsilon\|_{L^\infty(B_R(x_0))} \leqslant \|g\|_{L^\infty(\overline{B}_{2R}(x_0) \times I_u)}.$$
		\end{itemize}
		
		\textbf{Step 2.} Use a cutoff  technique and a cover argument to get the conclusion.
		Consider a cutoff function \( \phi^R \in C_0^\infty(\mathbb{R}^n) \) satisfying
		\begin{equation}\label{cutoff function}
			\phi^R \equiv 1 \text{ on } B_{3R/2}(x_0), \quad \mathrm{supp}(\phi^R) \subset B_{2R}(x_0), \quad 0 \leqslant \phi^R \leqslant 1.
		\end{equation}
		We define \( v_\varepsilon := \phi^R u_\varepsilon \),  then   $v_\varepsilon$ satisfies  
		\[
		-\Delta v_\varepsilon + (-\Delta)^s v_\varepsilon = \psi_\varepsilon \quad \text{in } V_{3\rho/4},
		\]
		where 
		\[
		\psi_\varepsilon := g_\varepsilon(x,u) + \Delta(u_\varepsilon(1-\phi^R)) - (-\Delta)^s(u_\varepsilon(1-\phi^R))
		\]
		is bounded, with the estimate
		\[
		R^2 \|\psi_\varepsilon\|_{L^\infty(B_R(x_0))} \leqslant C(n, s, \rho) \left( \|g_\varepsilon\|_{L^\infty(B_R(x_0))} + \|u_\varepsilon\|_{L^\infty(\mathbb{R}^n)} \right).
		\]
		
		Using interpolation inequalities and the result from \cite[Proposition 2.18]{FRRO22}, we derive a priori bounds for the \( C^{1,\alpha} \)-norm of \( v_\varepsilon \)  in small balls. Specifically, for every \( x_0 \in V_{\rho/4} \) and \( \delta > 0 \), there exists \( C_\delta > 0 \) such that
		\begin{align*}
			&\quad |u_{\varepsilon}|'_{1,\alpha; B_{R/2}(x_0)} = |v_{\varepsilon}|'_{1,\alpha; B_{R/2}(x_0)} \\
			&\leqslant C_{n,s,\alpha,\rho} \left( R^2 \|\psi_{\varepsilon}\|_{L^\infty(B_R(x_0))} + \delta |u_{\varepsilon}|'_{1,\alpha; B_{2R}(x_0)} + C_\delta \|u_{\varepsilon}\|_{L^\infty(B_{2R}(x_0))} \right) \\
			&\leqslant C_{n,s,\alpha,\rho} \left( \|g\|_{L^\infty(V_{3\rho/4} \times I_u)} + C_\delta \|u\|_{L^\infty(\mathbb{R}^n)} + \delta |u_{\varepsilon}|'_{1,\alpha; B_{2R}(x_0)} \right)
		\end{align*}
		for every \( R \in (0, \rho/10) \), \(\varepsilon \in (0, R) \).
		
		 From~\cite[Proposition 4.2]{SVWZ25}, it follows that 
		\begin{equation}\label{vsvsdv}
			\begin{split}
				\|u_{\varepsilon}\|_{C^{1,\alpha}(\bar{B}_{\rho/40}(y))}&\leqslant C\left( \|g\|_{L^\infty(V_{3\rho/4} \times I_u)} +  \|u\|_{L^\infty(\mathbb{R}^n)} \right)\\
				&\leqslant C\left( \|f\|_{L^\infty(V_{3p/4} \times I_u)} + \left(\|a\|_{L^{\infty}(\Omega)} + 1 \right) \|u\|_{L^\infty(\mathbb{R}^n)} \right)
			\end{split}
		\end{equation}
		for every $y\in V$, where the constant $C>0$ depends on $n,s,\alpha,\rho.$
		
		Owing to the Arzel\`a-Ascoli theorem and the covering argument,   we obtain that
		\begin{equation}\label{C1 norm estimate}
			\begin{aligned}
			 \|u\|_{C^{1,\alpha}(\overline{V})} &\leqslant C \left( \|f\|_{L^\infty(V_{3p/4} \times I_u)} + \left(\|a\|_{L^{\infty}(\Omega)} + 1 \right) \|u\|_{L^\infty(\mathbb{R}^n)} \right)\\
				&\leqslant C_{n,s,\alpha,\rho,\|a\|_{L^{\infty}(\Omega)}} \left( \|f\|_{L^\infty(V_{3p/4} \times I_u)} + \|u\|_{L^\infty(\mathbb{R}^n)} \right). \quad \qedhere
			\end{aligned}
		\end{equation} 
		
\end{proof}

\begin{Theorem}[Interior $C^{2, \alpha}$-regularity]\label{Thm: Interior C2-regularity}
    Suppose $u \in \mathcal { X } ^ { 1 , 2 } ( \Omega )$ is a bounded weak solution of 
    \[-\Delta u + (-\Delta)^{s} u+ a(x) u =f(x,u)  \quad \text{in }\Omega, \]
     where $a(x) \in  L^{\infty}(\Omega) \cap C^{\alpha}_{loc}(\Omega)$ and $f(x,t) \in C^{\alpha}_{loc}(\Omega \times \mathbb{R})$. Assume $V$ is an open domain with $ V \subset \subset \Omega$. Then, $u \in C^{2, \alpha}(\bar{V})$ for any $\alpha \in (0, 1)$.
\end{Theorem}

\begin{proof}
	The proof follows a suitable truncation method combined with interior $C^{1,\alpha}$-regularity argument, extending \cite[Theorem 1.5]{SVWZ25}. For the reader's convenience, we outline the key steps.


	\textbf{Step 1.} Regularization.
	For $0 < \varepsilon < R$, the mollified functions satisfy
	\[ -\Delta u_{\varepsilon} + (-\Delta)^s u_{\varepsilon} = g_{\varepsilon} \quad \text{in } V_{3\rho/4} \]   
	with $u_{\varepsilon} \in C^{2, \alpha}(\overline{V}_{3\rho/4}) \cap L^{\infty}(\mathbb{R}^n)$. 
	
		In view of  Theorem~\ref{Thm: Interior C1-regularity}, we can infer that 
	$g(x,u) \in C^{\alpha}_{\mathrm{loc}}(\Omega \times \mathbb{R})$.  More specifically, one has that
	$$ \|g_{\varepsilon}(\cdot, u(\cdot))\|_{C^{\alpha}(\overline{B}_R(x_0))} \leqslant \|g\|_{C^{\alpha}(\overline{B}_{2R}(x_0) \times I_u)} \left(1 + \|Du\|_{L^{\infty}(B_{2R}(x_0))}\right), \forall x_0 \in V_{\rho/4}.$$

	\textbf{Step 2.} Local estimate via cutoff argument.    
Let us denote $v_{\varepsilon} := \phi^R u_{\varepsilon}$, we obtain that
	\[ -\Delta v_{\varepsilon} + (-\Delta)^s v_{\varepsilon} = \psi_{\varepsilon} \quad \text{in } V_{3\rho/4}, \]  
	where $\phi^R$ is as in~\eqref{cutoff function}. In particular, 
	\[ R^2 |\psi_{\varepsilon}|_{0, \alpha; B_R(x_0)}' \leqslant C(n, s, \rho) \left( R^2 |g_{\varepsilon}(\cdot, u(\cdot))|_{0, \alpha; B_R(x_0)}' + \|u_{\varepsilon}\|_{L^{\infty}(\mathbb{R}^n)} \right). \]  
	
	\textbf{Step 3.} Compactness via Arzel\`a-Ascoli theorem.
	Using~\cite[Theorem 4.6]{GTbook} and the interpolation inequalities, we derive the $C^{2, \alpha}$-norm of $v_{\varepsilon}$  in small balls. Specifically, for every $x_0 \in V_{\rho/4}$ and $\delta>0$, there exists $C_\delta$ such that
	\begin{align*}
		&\quad |u_\varepsilon|_{2,\alpha; B_{R/2}(x_0)}' = |v_\varepsilon|_{2,\alpha; B_{R/2}(x_0)}' \\
		&\leqslant C_{n,s,\alpha,\rho} \left( R^2 |\psi_\varepsilon|_{0,\alpha; B_R(x_0)}' + \delta |u_\varepsilon|_{2,\alpha; B_{2R}(x_0)}' + C_\delta \|u\|_{L^\infty(B_{2R}(x_0))} \right) \\
		&\leqslant C_{n,s,\alpha,\rho} \left( \|g\|_{C^\alpha(\overline{V}_{3\rho/4} \times I_u)} \left( 1 + \|Du\|_{L^\infty(V_{3\rho/4})} \right) + C_\delta \|u\|_{L^\infty(\mathbb{R}^n)} +\, \delta |u_\varepsilon|_{2,\alpha; B_{2R}(x_0)}' \right),
	\end{align*}
	for every \( R \in (0, \rho/10) \) and \( \varepsilon \in (0, R) \).

	In the light of~\cite[Proposition 5.2]{SVWZ25} 
and the Arzel\`a-Ascoli theorem, the sequence $\{u_{\varepsilon}\}$ converges (up to a subsequence) to $u$ in $C^{2, \alpha}(\overline{V})$, which implies $u \in C^{2, \alpha}(\overline{V})$. More precisely,
	\begin{align*}
		&\quad \|u\|_{C^{2,\alpha}(\overline{V})} \leqslant C_{n, s, \alpha, \rho} \left( \|g\|_{C^{\alpha}(\overline{V}_{3\rho/4} \times I_u)} \left( 1 + \|Du\|_{L^\infty(V_{3\rho/4})} \right) + \|u\|_{L^\infty(\mathbb{R}^n)} \right) \\
		&\leqslant C_{n, s, \alpha, \rho} \left( \|g\|_{C^{\alpha}(\overline{V}_{3\rho/4} \times I_u)} \left( 1 + \|g\|_{L^\infty(\overline{V}_{7\rho/8} \times I_u)} + \|u\|_{L^\infty(\mathbb{R}^n)} \right) + \|u\|_{L^\infty(\mathbb{R}^n)} \right) \\
		&\leqslant C_{n, s, \alpha, \rho} \left( \|u\|_{L^\infty(\mathbb{R}^n)} + \|g\|_{C^{\alpha}(\overline{V}_{7\rho/8} \times I_u)} \right) \left( 1 + \|g\|_{C^{\alpha}(\overline{V}_{7\rho/8} \times I_u)} \right).
	\end{align*}
    Thus, $u \in C^{2, \alpha}(\bar{V})$ for any $\alpha \in (0, 1)$.
\end{proof}

\begin{Theorem}[$C^{2, \alpha}$-regularity up to boundary]\label{Thm: C^2 regularity up to boundary}
	Let $s\in(0, 1/2)$ and $\alpha \in (0,1)$ be such that $\alpha + 2s \leqslant 1$.
	Assume $\partial \Omega$ is of class $C^{2, \alpha}$. Suppose $u \in \mathcal { X } ^ { 1 , 2 } ( \Omega )$ is a weak solution of \eqref{problem}. If $a(x) \in C^{\alpha}(\bar\Omega) $ and $f \in C^{\alpha}(\bar\Omega \times \mathbb{R})$ satisfies (H1), then $u \in C^{2, \alpha}(\bar\Omega)$.
\end{Theorem}
\begin{proof}
	Let $u \in \mathcal { X } ^ { 1 , 2 } ( \Omega )$ is a weak solution of \eqref{problem}. Theorem~\ref{Thm: boundedness} implies $u \in L^{\infty}(\Omega)$. Using the boundedness of continuous functions on closed domain, $a(x)u + f(x,u) \in L^{\infty}(\Omega)$.
	By a similar argument in \cite[Theorem 1.2]{SVWZ22}, we obtain $C^{1, \alpha}$-regularity up to boundary: $u \in C^{1, \alpha}(\bar{\Omega})$ for any $\alpha \in (0, 1)$. 
	
	The $C^{2, \alpha}$-regularity up to boundary follows by a proof similar to \cite[Theorem 1.6]{SVWZ25}. For reader's convenience, we give a sketch of the proof:
	
	Step 1. Denote  $\mathcal{C}^{2,\alpha}(\bar \Omega) := \left\{ u \in C(\mathbb{R}^{n}): u \equiv 0 \text{ in } \mathbb{R}^{n} \backslash \Omega, u|_{\Omega} \in C^{2, \alpha}(\bar\Omega) \right\}$ and $\mathcal{L}_{t} := (1-t)(- \Delta) + t \mathcal{L}$. 
	Note that for any $t \in [0,1]$, $\mathcal{L}_{t}$ is a bounded linear operator from $\mathcal{C}^{2, \alpha}(\bar\Omega)$ to $C^{\alpha}(\bar \Omega)$. Since $\mathcal{L}_{0}= -\Delta$ is surjective, applying the continuity method, we deduce that, for every $g \in C^{\alpha}(\bar \Omega)$ there exists a unique $v \in \mathcal{C}^{2, \alpha}(\bar\Omega)$ such that $\mathcal{L} v = g$ a.e. in $\Omega$. 
	
	Step 2.  Using \textit{Lax-Milgram Theorem}  to  bilinear mapping $B_{s}[u,v]$ and bounded linear functional $\bar f_{a}: \mathcal{X}^{1,2}(\Omega) \rightarrow \mathbb{R}$
	\[v \mapsto \int_{\Omega} -a(x)uv ~dx + \int_{\Omega} f(x,u)v ~dx\]
	where $u \in \mathcal{X}^{1,2}(\Omega)$ is a weak solution, we deduce that the unique solution $u \in \mathcal{C}^{2,\alpha}(\bar \Omega)$.
\end{proof}

\begin{remark}\label{remark: restriction s}
	The restriction \(s \in (0, 1/2)\) and $\alpha\in(0,1)$ satisfying \(\alpha + 2s \leqslant 1\) in Theorem \ref{Thm: C^2 regularity up to boundary} is sharp. 
	
	\textbf{1. $s\in(0, 1/2)$ is unavoidable.}
	Even though the Laplacian dominates in local smoothness (see Theorems 4.3 and 4.4), the nonlocality of the fractional Laplacian affects the \(C^{2, \alpha}\)-regularity up to the boundary and such effect cannot be ignored for \(s \geqslant 1/2\). We give a counterexample below to show $s\in(0, 1/2)$ is unavoidable.
	
	
	\textbf{2. \(\alpha + 2s \leqslant 1\) is essential.}
	The condition \(\alpha + 2s \leqslant 1\) ensures compatibility between the H\"older exponent \(\alpha\) and the fractional order \(s\). 
	The fractional Laplacian \((-\Delta)^s\) introduces a weak singularity with a regularity loss of order \(2s\). 
	
	Our proof of Theorem 1.6 is based on  \cite[Lemma 5.3]{SVWZ25}, whose proof  explicitly uses \(s \in (0, 1/2)\) and $\alpha\in(0,1)$ satisfying \(\alpha + 2s \leqslant 1\) to bound the contribution of the nonlocal term, confirming that these condition are essential.
\end{remark}

\begin{Example}[A counterexample to Remark \ref{remark: restriction s}]\label{example}
		
		Consider the mixed local-nonlocal elliptic equation 
		\begin{equation}\label{eq: example}
			\begin{cases}
				-\Delta u + (-\Delta)^s u + au= f(x, u) & \text{in } (0,1), \\
				u = 0 & \text{in } \mathbb{R} \setminus (0,1).
			\end{cases}
		\end{equation}
		When $s\in (1/2,1)$, the solution $u$ fails to attain $C^2$ regularity at the boundary point 0.
	\end{Example}

		\begin{proof}
			We proceed by contradiction. Assume that the solution \(u\) behaves near the boundary point \(0\) as
			\[ u(x) = Ax + O(x^2) \quad (x \to 0^+), \]
			where \(A \neq 0\) is a positive constant.
			
			We first claim that $\lim_{x\to 0^+}  (-\Delta)^s u(x)= + \infty$. 
			Since $u(x)=0$ outside $(0,1)$,  we separate the integral
			\[ (-\Delta)^s u(x)	= c_{1,s} \, \text{P.V.} \int_{\mathbb{R}} \frac{u(x) - u(y)}{|x - y|^{1 + 2s}} dy, \]
			into the interior region $(0,1)$ and the exterior region $(-\infty,0)\cup(1,+\infty)$.
			
			For $y \in (0,1)$, by Theorem 4.4, we have $u(y) \in C^2$ when $0 \ll y < 1$. Therefore, it suffices to consider the case $0 < y < x$. 
			\begin{align*}
				&\quad \lim_{x\to 0^+}  \int_0^x \frac{u(x)-u(y)}{|x-y|^{1+2s}}dy\\
				&=\lim_{x\to 0^+}  \int_0^x \frac{\left(Ax+O(x^2)\right)- \left(Ay+O(y^2)\right)}{(x-y)^{1+2s}}dy\\
				&=\lim_{x\to 0^+}  \int_0^x \frac{A(x-y)+O(x^2-y^2)}{(x-y)^{1+2s}}dy\\
				&=\lim_{x\to 0^+}  \int_0^x \frac{A}{z^{2s}}dz + \lim_{x\to 0^+}  \int_0^x \frac{O(x^2-y^2)}{(x-y)^{1+2s}}dy= +\infty.
		   \end{align*}
		
		    For  $y \in (-\infty,0)\cup(1,+\infty)$, $u(y)=0$. Thus,
		    \begin{align*}
		        &\quad \lim_{x\to 0^+}  \left( \int_{-\infty}^0 + \int_{1}^{\infty}\right)  \frac{u(x)-u(y)}{|x-y|^{1+2s}}dy\\
		        &= \lim_{x\to 0^+}  \left( Ax + O(x^2)\right)\left( \int_{-\infty}^0 \frac{dy}{|x-y|^{1+2s}}+ \int_{1}^{\infty} \frac{dy}{|x-y|^{1+2s}} \right)\\
		        &=  \lim_{x\to 0^+}  \left( Ax + O(x^2)\right) \left( \int_x^\infty \frac{dz}{z^{1+2s}} + \int_{1}^{\infty} \frac{dy}{(y-x)^{1+2s}} \right)\\
		        &= \frac{1}{2s} \lim_{x\to 0^+}  \left( Ax + O(x^2)\right) \left( x^{-2s} + (1-x)^{-2s} \right)= +\infty.
		    \end{align*}
	        We have thus proven the Claim.
	        
	        The equation \eqref{eq: example} can be written as
	        $ -\Delta u = -au + f(x, u) - (-\Delta)^s u$.
	        When \(x \to 0^+\), the left-hand side \(-\Delta u = -u''(x) = O(1)\), but the right-hand side, if \(s > 1/2\), 
	        $-au + f(x, u) - (-\Delta)^s u \to -\infty$. This leads to a contradiction. Therefore, the assumption that the solution has \(C^{2}\) regularity (i.e., \(u'' = O(1)\) is bounded) when \(s > 1/2\) is invalid.
		\end{proof}		

\begin{Corollary}
	Under the assumption of Theorem~\ref{Thm: C2-regularity up to boundary}, assume $f$ satisfies (H2)-(H4). Then there exists a classical solution $u \in C^{2, \alpha}(\bar\Omega)$.
\end{Corollary}

Before ending this section, as a corollary of \cite[Theorem 1.1]{BVDV21}, we obtain the radial symmetry of non-negative weak solution.
\begin{Theorem}\label{Thm: symmetric}
	Assume that $\Omega$ is symmetric and convex with respect to the hyperplane $\left\{ x _ { 1 } = 0 \right\} $, $\partial  \Omega$ is of class $C^{1}$ and $a(x) \in L^{\infty }(\Omega)$.	
	If $ 0 \leqslant u \in C ( \mathbb { R } ^ { n } )$ is a weak solution of \eqref{problem}, then u is symmetric with respect to $\left\{ x _ { 1 } = 0 \right\}$ and strictly increasing in the $x _ { 1 }$ direction in $\Omega \cap \left\{ x _ { 1 } < 0 \right\} .$ 
\end{Theorem}

\bibliographystyle{alpha}

\bibliography{reference}

\end{document}